\documentclass[11pt]{article}

\usepackage{setspace}
\setdisplayskipstretch{1.5}

\usepackage[colorlinks=true,linkcolor=blue,citecolor=blue, urlcolor=black]{hyperref}

\usepackage{amsmath, amsthm, amssymb, fullpage,  xcolor}
\usepackage{bbm}
\usepackage{mathrsfs}
\usepackage{thmtools}
\usepackage{enumerate}
\usepackage{appendix}
\usepackage{mathtools}
\usepackage[all]{xy}
\usepackage{cleveref}
\crefformat{section}{§#2#1#3}
\Crefformat{section}{§#2#1#3}
\crefname{observation}{observation}{observations}
\Crefname{observation}{Observation}{Observations}
\usepackage{mleftright}
\mleftright
\usepackage{nicefrac}
\AddToHook{bfseries}{\boldmath}

\usepackage[backend=biber,style=alphabetic]{biblatex}
\DeclareLabelalphaNameTemplate{%
  \namepart[use=true, pre=true, strwidth=1, compound=true]{prefix}%
  \namepart[strwidth=1, compound=true]{family}%
} % :contentReference[oaicite:0]{index=0}

\DeclareLabelalphaTemplate{%
  \ifnumgreater{\value{labelname}}{3}
    {\labelelement{\field[strwidth=1,strside=left]{labelname}}%
     \labelelement{\literal{+}}}
    {\labelelement{\field[strwidth=3,strside=left]{labelname}}}%
  \labelelement{\literal{'}}%
  \labelelement{\field[strwidth=2,strside=right]{year}}%
}

\newcommand{\trace}{\mathrm{Tr}}
\newcommand{\tr}{\textnormal{tr}}

\newcommand{\calH}{\mathcal{H}}

\newcommand{\calO}{\mathcal{O}}

\newcommand{\calU}{\mathcal{U}}

\newcommand{\calI}{\mathcal{I}}

\newcommand{\bC}{\mathbb{C}}

\newcommand{\Span}{\mathrm{Span}}

\newcommand{\E}{\mathbb{E}}

\newcommand{\diag}{\mathrm{diag}}

\newcommand{\R}{\mathbb{R}}

\newcommand{\bR}{\mathbb{R}}

\newcommand{\var}{\mathrm{Var}}

\newcommand{\calS}{\mathcal{S}}

\newcommand{\abs}[1]{\left\lvert#1\right\rvert}

\newcommand{\heat}{e^{-\frac{t}{2} \partial_x^2}}

\newcommand{\disc}{\mathrm{Disc}}

\allowdisplaybreaks

\newcommand{\calV}{\mathcal{V}}

\newcommand{\bone}{\mathbbm{1}}

\newcommand{\cone}{K_{\lilconv}}
\newcommand{\bcone}{\partial K_{\lilconv}}

\newcommand{\palpha}{\mathbf{\mathrm{E}}_\beta[\alpha|\gamma]}
\newcommand{\pbeta}{\mathbf{\mathrm{E}}_\alpha[\beta|\gamma]}

\newcommand{\jac}{\mathbf{\mathrm{J}}_{\scriptscriptstyle \boxplus_n}}

\newcommand{\roots}{\Omega_{\scriptscriptstyle\boxplus_n}}
\newcommand{\rootsi}{\Omega_{\scriptscriptstyle\boxplus_n}^i}
\newcommand{\dermap}{ \Omega_{\scriptscriptstyle\partial_x}}
\newcommand{\dermapi}{\Omega_{\scriptscriptstyle \partial_x}^{i}}
\newcommand{\fder}{f_{\partial_x}}
\newcommand{\lilder}{\scriptscriptstyle\partial_x}
\newcommand{\lilconv}{\scriptscriptstyle\boxplus_n}

\newcommand{\pmap}[1]{\mathrm{Poly}[#1]}

\newcommand{\dd}{\partial}

\newcommand{\score}{\mathscr{J}_n}
\newcommand{\scoreder}{\mathscr{J}_{n-1}}

\newcommand{\jacder}{\mathbf{\mathrm{J}}_{\scriptscriptstyle \partial_x}}

\newcommand{\hessd}{\mathbf{\mathrm{H}}_{\scriptscriptstyle \partial_x}}
\newcommand{\hessder}{\mathbf{\mathrm{H}}_{\scriptscriptstyle \partial_x}^{i}}

\newcommand{\hessconv}{\mathbf{\mathrm{H}}_{\scriptscriptstyle\boxplus_n}^{i}}

\newcommand{\ga}{\alpha}     %lowercase alpha
\newcommand{\eps}{\varepsilon} %epsilon
\newcommand{\gl}{\lambda}    %lowercase lambda
\newcommand{\sN}{\mathscr{N}}

\newcommand{\lpr}[1]{\left(#1\right)}
\newcommand{\lbr}[1]{\left[#1\right]}

\newcommand{\rox}{\rho_X}
\newcommand{\roy}{\rho_Y}
\newcommand{\roz}{\rho_Z}

\newcommand{\mult}{\mathrm{mult}}

\newcommand{\fconv}{f_{\boxplus_n}}

\expandafter\let\expandafter\originald\csname\encodingdefault\string\d\endcsname
\DeclareRobustCommand*\d
    {\ifmmode\mathop{}\!\mathrm{d}\else\expandafter\originald\fi}

\DeclareRobustCommand{\rchi}{{\mathpalette\irchi\relax}}
\newcommand{\irchi}[2]{\raisebox{\depth}{$#1\chi$}} % inner command, used by \rchi

\newcommand{\norm}[1]{\left\|#1\right\|}

\DeclareMathOperator{\He}{He}
\DeclareMathOperator{\Disc}{Disc}

\newcommand{\adj}[1]{{#1}^\top}

\newcommand{\calpha}{\overset{\circ}{\alpha}}
\newcommand{\ceta}{\overset{\circ}{\beta}}
\newcommand{\cdelta}{\overset{\circ}{\delta}}
\newcommand{\camma}{\overset{\circ}{\gamma}}

\newcommand{\mesh}{\mathrm{mesh}}

\newcommand{\slambda}{\sqrt{\lambda}\,_*}
\newcommand{\sone}{\sqrt{1-\lambda}\,_*}

\newcommand{\pdelta}{\mathbf{\mathrm{E}}[\alpha|\delta]}
\newcommand{\footremember}[2]{%
    \footnote{#2}
    \newcounter{#1}
    \setcounter{#1}{\value{footnote}}%
}
\newcommand{\footrecall}[1]{%
    \footnotemark[\value{#1}]%
} 

\title{Finite Free Information Inequalities}

\author{Jorge Garza-Vargas\footremember{pu}{MIT} \and Nikhil Srivastava\footremember{uc}{UC Berkeley} \and Zachary Stier\footrecall{uc}}

\date{\today}

\begin{document}

\maketitle

\begin{abstract}
    We develop finite free information theory for real-rooted polynomials, establishing finite free analogues of entropy and Fisher information monotonicity, as well as the Stam and entropy power inequalities. These results resolve conjectures by Shlyakhtenko and Gribinski and recover inequalities in free probability in the large-degree limit. Equivalently, our results may be interpreted as potential-theoretic inequalities for the zeros of real-rooted polynomials under differential operators which preserve real-rootedness. Our proofs leverage a new connection between score vectors and Jacobians of root maps, combined with convexity results for hyperbolic polynomials. We further characterize the equality cases in our inequalities, which arise from Hermite polynomials.
    
\end{abstract}

\numberwithin{equation}{section}
\newtheorem{theorem}{Theorem}[section]
\newtheorem{definition}[theorem]{Definition}
\newtheorem{conjecture}[theorem]{Conjecture}
\newtheorem{lemma}[theorem]{Lemma}
\newtheorem{corollary}[theorem]{Corollary}
\newtheorem{proposition}[theorem]{Proposition}
\newtheorem{observation}[theorem]{Observation}
\newtheorem{remark}[theorem]{Remark}
\newtheorem{problem}[theorem]{Problem}
\newtheorem{assumption}[theorem]{Assumption}
\newtheorem{claim}[theorem]{Claim}
\newtheorem{fact}[theorem]{Fact}
\newtheorem{example}[theorem]{Example}

\maketitle
\tableofcontents
\section{Introduction}
\label{sec:introduction}
Let $\R_{\le n}[x]$ denote the real vector space of univariate polynomials of degree at most $n$.
This paper is motivated by the following question:
\begin{quote}
    {\bf  Question 1.} Suppose $T:\R_{\le n}[x]\rightarrow \R_{\le n}[x]$ is a differential operator  which preserves real-rootedness. For a real-rooted polynomial $p(x)$, how is
    the spacing of the zeros of $T(p)$ related to the spacing of the zeros of $p$?
\end{quote}
This question sits at the intersection of two lines of inquiry. The first is the P\'olya--Schur program, which seeks to understand which linear transformations preserve the property of not having zeros in a prescribed set --- in this case the complex upper half plane. Question 1 asks for quantitative information about the spacing of these zeros. The question of spacing goes back to M. Riesz \cite{stoyanoff1925theoreme}, who showed that if $p(x)$ is real-rooted, then the minimum gap between the zeros of $p'(x)$ is at least the minimum gap of $p(x)$. Riesz's result generalizes to arbitrary real-rootedness preserving $T$ by the Hermite--Poulain theorem \cite[p. 4]{obreschkoff1963verteilung} (see also \cite{katkova2011multiplier}). In this paper, we study averaged notions of spacing, namely the discriminant and a certain Coulomb potential of the zeros of $T(p)$, and prove sharp inequalities regarding them.

The second is the nascent area of finite free probability initiated in \cite{marcus2022finite, marcus2021polynomial}, which studies the expected characteristic polynomials of certain random matrices and their analogies with Voiculescu's free probability theory. It turns out that the basic operations in this theory correspond to certain differential operators $T$ as above. D.\ Shlyakhtenko suggested that there should be an analogue of ``free information theory'' in the setting of finite free probability where the convolution is the finite free convolution --- in particular, there should be analogues of the classical and free entropic inequalities. Shlyakhtenko \cite{shlyakh15} and later A.\ Gribinski \cite{gribinski2019notion} conjectured specific forms of the finite free Stam and entropy power inequalities, which in the framework of Question 1 correspond precisely to bounds on the Coulomb potential and discriminant of $T(p)$ in terms of that of $p$. 

In this paper we develop a framework for answering Question 1 which is informed by both lines of work, and use it to prove several results including the conjectures of Shlyakhtenko and of Gribinski. Let us introduce a few necessary definitions and notation. Let  $p(x)=\prod_{i=1}^n (x-\ga_i)$ be a monic real-rooted polynomial. As is customary in finite free probability, we will draw analogies from thinking of polynomials as probability distributions, and define the moments and variance of $p$ as 
\begin{equation}m_k(p):=m_k(\ga):=\frac{1}{n}\sum_{i=1}^n\ga_i^k \qquad \text{and} \qquad  \var(p):=\var(\ga):=m_2(p)-m_1(p)^2.\end{equation} 
We now introduce the main objects of study.

\begin{definition}[Score, Fisher information, and entropy]
\label{def:finite_free_entropies}
Given $\alpha\in \bR^n$, we define the {\em score vector} of $\alpha$ as
\[\score(\ga):=\lpr{\sum\limits_{j : j\neq i}\frac{1}{\ga_i-\ga_j}:1\le i\le n}\in \bR^n.\]
Then, if $p(x) = \prod_{i=1}^n(x-\alpha_i)$,  define the {\em finite free Fisher information} of $p(x)$ through the norm of the (normalized) score vector of $\alpha$ as:
\begin{equation}\label{eqn:ffdef}
   \Phi_n(p):=\frac{1}{n}\norm{\frac{2}{n-1} \score(\ga)}^2=\frac{1}{n} \sum_{i=1}^n  \left(\frac{2}{n-1} \sum_{j\neq i}\frac{1}{\ga_i-\ga_j}\right)^2.
\end{equation}
Lastly, define the {\em finite free entropy} of $p(x)$ as 
\[\rchi_n(p):=\frac{2}{n(n-1)}\sum\limits_{i<j}\log\abs{\ga_i-\ga_j}.\]
\end{definition}
\begin{remark}[Discriminant and Coulomb potential]
For monic $p(x)$, $\rchi_n(p)=\frac{1}{2\binom{n}{2}}\log\Disc(p)$ where $\Disc(p):=\prod\limits_{i<j}(\ga_i-\ga_j)^2.$ Moreover, by expanding the square in \eqref{eqn:ffdef} and observing that cross terms cancel, we see that $\Phi_n(p)$ is a scaling of $\sum_{i<j}(\ga_i-\ga_j)^{-2}$.
\end{remark}
It will be important to understand how these quantities behave under rescaling. For $c\in\R\setminus \{0\}$, write $c_*p$ for the polynomial obeying $(c_*p)(x):=c^np\lpr{\frac{x}{c}}$; if $c=0$, then we take the convention that $(c_*p)(x)=x^n$. Importantly, the roots of $c_*p$ can be obtained by multiplying the roots of $p$ by the factor  $c$. 

The quantities defined above  obey the following elementary rescaling laws, which are familiar from classical information theory. 
\begin{equation}
\var(c_*p)=c^2\var (p),\qquad
    \rchi_n(c_*p)=\rchi_n(p)+\log\abs{c}, \qquad\text{and}\qquad \Phi_n(c_*p) = \frac{1}{c^2} \Phi_n(p). \label{eq:rescaling}
\end{equation}
And, specifically $\rchi_n, \Phi_n$, and $\score$ are the finite analogues of important quantities in free probability theory, from which they inherit their name, normalization, and notation (see \Cref{sec:inspiration_free_probability} for a discussion), and from where inspiration was drawn to conjecture the results we state below.   
\subsection{Statement of results}
First, we consider the simplest differential operator, $T=\partial_x$, and prove the following. 
\begin{theorem}[Finite free Fisher information monotonicity]
\label{thm:fisher_monotonicity}
    Let $p(x)$ be a real-rooted polynomial of degree $n$. If $\widetilde{p'}(x)$ is the rescaling of $p'(x)$ satisfying $\var\lpr{\widetilde{p'}}=\var(p)$, then
    $$\Phi_{n-1}\lpr{\widetilde{p'}} \le \Phi_n(p). $$
\end{theorem}
\Cref{thm:fisher_monotonicity} is sharp for Hermite polynomials. Below, we provide two proofs of this theorem. First, in \Cref{sec:matrix_analysis_fisher_monotonicity} we give a proof via matrix analysis that elucidates some interesting connections with the Gauss--Lucas matrix and differentiator matrices. Second, in \Cref{sec:geometry_of_polynomials_fisher_monotonicity} we give a more conceptual proof which proceeds by viewing the polynomial $p'(x)$ as a univariate restriction of a hyperbolic polynomial and then invokes a convexity result of H.\ Bauschke et al. \cite{bauschke2001hyperbolic}. 

Next, we consider the {\em finite free convolution} of two polynomials. Given a degree $n$ monic polynomial $p(x)$, let $\hat{p}(\cdot)$ be any  polynomial such that $p(x)=\hat{p}(\partial_x)x^n$. Then, define
\begin{equation}
    p(x)\boxplus_n q(x):= \hat{p}(\partial_x)\hat{q}(\partial_x)x^n = \hat{p}(\partial_x)q(x) = \hat{q}(\partial_x)p(x).
\end{equation}
It was shown by J.\ Walsh \cite{walsh1922location} that $\boxplus_n$ preserves real-rootedness, and noted by Br\"{a}nd\'{e}n, Krasikov, and  Shapiro \cite{branden2016elements} and by Leake and Ryder \cite[Proposition 1.2]{leake2018further} that every differential operator $T$ which preserves real-rootedness can be written as $T=\hat{q}(\partial_x)$ for some real-rooted $q$. 

We show the following, which was conjectured by Shlyakhtenko \cite{shlyakh15}.
\begin{theorem}[Finite free Stam inequality]
\label{thm:stams_inequality}
    Let $p(x)$ and $q(x)$ be any two real-rooted polynomials of degree $n$ and $\lambda\in [0, 1]$. Then 
    $$  \Phi_n(\sqrt{\lambda}\,_* p \boxplus_n \sqrt{1-\lambda }\,_* q)\le \lambda \Phi_n(p) + (1-\lambda)\Phi_n(q).$$
\end{theorem}

\begin{remark}
The alternative form of Stam's inequality
    $$\frac{1}{\Phi_n(p)} + \frac{1}{\Phi_n(q)} \leq \frac{1}{\Phi_n(p\boxplus_n q)}$$
    can easily be derived from Theorem \ref{thm:stams_inequality} by choosing $\lambda = \frac{\Phi_n(p)^{-1}}{\Phi_n(p)^{-1}+ \Phi_n(q)^{-1}}$, defining $\tilde{p} := \lambda^{-1/2}_* p$ and $\tilde{q} := (1-\lambda)^{-1/2} q$, and applying the inequality in Theorem \ref{thm:stams_inequality} to $\tilde{p}$ and $\tilde{q}$ with this choice of $\lambda$. 
\end{remark}

The inequality is sharp for Hermite polynomials. The proof of \Cref{thm:stams_inequality} appears in \Cref{sec:proof_of_stam}. It rests on a key conceptual contribution of this paper, which is a mechanism for relating the score vectors of $p\boxplus_n q$ and $p$ via the Jacobian of the induced map from the roots of $p$ to the roots of $p\boxplus_n q$; this is developed in \Cref{sec:finite_scores_to_finite_scores} and its probabilistic interpretation as a coupling is discussed in \Cref{sec:double_stochasticity}. The other key ingredient, as in the derivative case, is to view $p\boxplus_n q$ as a restriction of a multivariate hyperbolic polynomial and appeal to a convexity result of Bauschke et al.\ \cite{bauschke2001hyperbolic}, which has its roots in L.\ G{\aa}rding's foundational work on hyperbolicity \cite{gaarding1951linear}.

Each of these main results has implications for the finite free entropy. Given \Cref{thm:fisher_monotonicity}, integrating along the heat flow yields a monotonicity result for the finite free entropy, as conjectured by the first-named author \cite{garvar25}. 
\begin{theorem}[Mononicity of finite free entropy under differentiation]\label{thm:monotonicity}
    Let $p(x)$ be a monic real-rooted polynomial of degree $n$. If $\widetilde{p'}(x)$ is the rescaling of $p'(x)$ satisfying $\var\lpr{\widetilde{p'}}=\var(p)$,  then the finite free entropy obeys
    \begin{equation}
    \label{eq:entropy_monotonicity}
    \rchi_n\lpr{p}+C_n\le\rchi_{n-1}\lpr{\widetilde{p'}},
        \end{equation}
    where $C_n:=\rchi_{n-1}\lpr{\check{H}_{n-1}}-\rchi_n\lpr{\check{H}_n}$ and $\check{H}_n(x)$ and $\check{H}_{n-1}(x)$ are  Hermite polynomials normalized so that $\var(\check{H_n}) = \var(\check{H}_{n-1})=1$. 
    Moreover, $C_n>0$, so we also find that
    \[\rchi_n(p)<\rchi_{n-1}\lpr{\widetilde{p'}}.\]
\end{theorem}
Inequality \eqref{eq:entropy_monotonicity} is tight for Hermite polynomials. We prove \Cref{thm:monotonicity} in \Cref{sec:entropy_monotonicity}. 

Similarly, given \Cref{thm:stams_inequality}, differentiating along the heat flow yields the finite free entropy power inequality conjectured by Gribinski \cite{gribinski2019notion}. Given $p(x)$ a degree $n$ real-rooted polynomial, we define 
\[\sN_n(p):=\exp\lpr{2\rchi_n(p)}.\]
We  prove the following entropy power inequality, conjectured by Gribinski \cite[Conjecture 1]{gribinski2019notion}. 

\begin{theorem}
    [Finite free entropy power inequality]\label{thm:ffepi}
    Let $p(x)$ and $q(x)$ be any monic real-rooted polynomials of degree $n$. Then, 
    \[\sN_n\lpr{p\boxplus_nq}\ge\sN_n(p)+\sN_n(q).\]
\end{theorem}

\begin{remark}[Brunn--Minkowski inequality]
    As noted by Gribinski \cite{gribinski2019notion}, one can rewrite the finite free entropy power inequality in terms of discriminants as
    $$\disc(p\boxplus_nq)^{\frac{1}{{n \choose 2}}} \ge \disc(p)^{\frac{1}{{n\choose 2}}} + \disc(q)^{\frac{1}{{n\choose 2}}},$$ 
    which is reminiscent of the Brunn--Minkowski theorem. 
\end{remark}

Theorem \ref{thm:ff7} will be derived, in the standard way, from the finite free version of Lieb's inequality (see Theorem \ref{thm:ff7} below), which in turn will be derived from the finite free Stam inequality given above.

\subsubsection{Equality cases}
\label{sec:intro_eq_cases}

In  \cref{sec:equality_cases_stam} and \cref{sec:equality_cases_epi} we will show that when $p(x)$ and $q(x)$ have simple roots up to shifts and rescaling,  Hermite polynomials are the only polynomials that achieve equality  for  the inequalities obtained in this paper. See Theorems \ref{thm:equality_cases_Fisher} and \ref{thm:equality_cases_Lieb_EPI} for precise statements. These same statements were proven for Stam's inequality and the entropy power inequality in the  independent and concurrent work of Hashemi and Hyun \cite{hashemi2026finite} (see the bibliographical note at the end of \cref{sec:introduction} for further details). 

The main geometric input used in our approach for characterizing the equality cases is the strict curvature  of the boundary of the hyperbolicity cone in question. Concretely, in the case of the derivative, we work with the polynomial 
$$g_{\lilder}(\zeta_1, \dots, \zeta_n)  :=  \sum_{i=1}^n \prod_{j : j \neq i} \zeta_j,$$
which is hyperbolic in the direction $\bone_n$, and has hyperbolicity cone $K(g_{\lilder}, \bone_n) = \{ \zeta\in \bR^n : \dermap^{n-1}(\zeta) > 0  \}.$ In this case,  a result of Renegar \cite{renegar2006hyperbolic} (see Theorem \ref{thm:renegar} below) implies that the boundary $\partial K (g_{\lilder}, \bone_n)$ is strictly curved at points with enough regularity. In other words,  the boundary of the hyperbolicity cone does not have any flat faces. Analogously, for the inequalities involving  the free convolution, we work with the polynomial
$$g_{\lilconv}(\xi_1, \dots, \xi_n \zeta_1,\dots, \zeta_n) :=  \sum_{\sigma\in S_n} \prod_{i=1}^n (\xi_i+\zeta_{\sigma(i)}),$$
which is hyperbolic in the direction $(\bone_n, 0_n)$, where $0_n$ denotes the all-zeroes vector in $\bR^n$. In this case, the hyperbolicity cone is given by $ K(g_{\lilconv}, (\bone_n, 0_n)) =\{ (\xi,\zeta)\in \bR^n\times \bR^n : \roots^{n}(\xi,\zeta) > 0  \}$. We then show that the boundary $\partial  K(g_{\lilconv}, (\bone_n, 0_n))$ is strictly curved at regular enough points (see Theorem \ref{thm:curved_cone_conv} below). To the best of our knowledge this is a new result that might be of independent interest.

\subsection{Related work}
The idea that differentiation evens out zero spacings goes back go the Gauss--Lucas theorem, whose proof uses the observation that for $p(x)$ with roots $\ga_1,\ldots,\ga_n$, one has
$$ \frac{p'(z)}{p(z)}=\sum_{i=1}^n\frac1{z-\ga_i},$$
which serves as an electrostatic model for the roots of $p'(z)$. 

The result of Riesz \cite{stoyanoff1925theoreme}, mentioned at the beginning, was improved by Walker \cite{walker1995bounds} who showed that the minimum zero spacing must increase by a factor of $1+O(1/n)$. Farmer and Rhoades \cite{farmer2005differentiation} extended Riesz's result to entire functions of order $1$.

Gribinski \cite{gribinski2019notion} showed that for all real-rooted monic $p$ and $q$ one has $\Disc(p\boxplus_n q)\ge \Disc(p)$, with equality only if $q(x)=(x-c)^n$ for some real $c$. In the special case of $T=1-c\partial_x$, this yields a particular instance of \Cref{thm:ffepi}. 

There has recently been a lot of interest in the behavior of the zeros of $p^{(cn)}(x)$, i.e., the case of repeated differentiation $T=(\partial_x)^{cn}$ \cite{steinerberger2023free, kiselev2022flow,hoskins2023dynamics,arizmendi2023finite}, which were shown to be related to the fractional free convolutions of a measure in free probability.
Given this connection, Theorems \ref{thm:fisher_monotonicity} and \ref{thm:monotonicity} are the finite free analogues of the monotonicity results of Shlyakhtenko and Tao for fractional free convolutions \cite[Theorem 1.6]{shlyakhtenko2022fractional}. Similarly, given the connections between the finite free convolution and the free convolution \cite{marcus2021polynomial, arizmendi2018cumulants}, \Cref{thm:stams_inequality} is the finite analog of Voiculescu's \cite{voiculescu1998analogues} free Stam inequality (see also  \cite{shlyakhtenko2007free}) and, \Cref{thm:ffepi}  is the finite analog of the free entropy power inequality of Szarek and Voiculescu \cite{szarek1996volumes}. 

Importantly, in \cref{sec:large_n_limit} we show that all of the aforementioned (infinite-dimensional) free probability inequalities can be  obtained by taking $n\to \infty$ in our results.  Thus, the results in this paper also provide fundamentally new proofs for the free probability theorems. See \Cref{sec:inspiration_free_probability} for more on the relation to free probability. 

In the related setting of finite difference operators (rather than differential operators), Br\"{a}nd\'{e}n, Krasikov, and Shapiro \cite{branden2016elements} proved an analogue of Riesz's theorem, which Leake and Ryder \cite{leake2020connecting} generalized to the setting of the $q-$multiplicative finite free convolution.

\subsection{Relation to free probability theory}
\label{sec:inspiration_free_probability}

In a sequence of influential papers \cite{voiculescu1993analogues, voiculescu1994analogues, voiculescu1996analogues, voiculescu1998analogues}, Voiculescu introduced the free probability analogues of Shannon entropy and Fisher information and  used them to solve problems in von Neumann algebras which, at the time, were considered to be out of reach. Since then, other notions of entropy have proliferated in the context of free probability and operator algebras giving rise to several interesting research directions \cite{dykema1997two, ge1998some, jung2007strongly, shlyakhtenko2022fractional, hayes2022random, jekel2024elementary, hayes2025consequences}. We refer the reader to Chapters 7 and 8 of \cite{mingo2017free} for an introduction to the subject. 

\paragraph{Free entropy.} In a nutshell, (microstates) free entropy is a quantitative measurement of how well infinite-dimensional operators can be approximated by finite matrices.  In the case of one operator, this amounts to looking at unitary orbits of the form
$\calO(D):=\{UDU^* : U\in \calU(n)\}\subset M_n(\bC),$
for real diagonal matrices $D=\diag(\alpha_1, \dots, \alpha_n)$, and measuring their intrinsic volume, which is given by  
$C_n \prod_{i\neq j} \lvert\alpha_i-\alpha_j\rvert^2,$
for some constant $C_n>0$ that depends only on $n$. After taking the logarithms of such volumes and normalizing, Voiculescu was led \cite{voiculescu1993analogues, voiculescu1994analogues} to define  the free entropy of a self-adjoint non-commutative random variable $x$ as 
\begin{equation}
\label{eq:free_entropy}
\rchi(x) := \iint \log\lvert s-t\rvert \d\mu_x(s)\d\mu_x(t),
\end{equation}
where $\mu_x$ is the spectral distribution of $x$. So, the quantity $\rchi_n(\cdot)$, introduced in \Cref{def:finite_free_entropies}, is the  analog of $\rchi(\cdot)$ obtained by replacing the notion of spectral distribution of an operator by that of root distribution of a polynomial, and where the infinite terms (corresponding to $s=t$ in the integral) have been omitted.\footnote{Interestingly, in his first paper on free entropy, Voiculescu stated that this (i.e.\ $\rchi_n$) is a natural quantity to consider (see \cite[Remark 7.3]{voiculescu1993analogues}).} 

\paragraph{Free Fisher information.} In \cite{voiculescu1993analogues} Voiculescu derived the notion of free Fisher information from (\ref{eq:free_entropy}), via the free analogue of the de Bruijn identity (see \Cref{sec:entropy_power} for the de Bruijn identity in the finite free setting). Later in \cite{voiculescu1998analogues} an alternative approach was given. To be precise, given a probability measure $\mu$ on $\bR$, define the Hilbert transform of $\mu$ as the Cauchy principal value
$$\calH_\mu(t):= \frac{1}{\pi}\, \text{p.v.}\int \frac{1}{t-s} d\mu(s). $$
Then, if $x$ is a self-adjoint non-commutative random variable in a $W^*$-probability space $(M, \tau)$, one can define its conjugate variable (which is the free analog of score function) as  $\xi:= 2\pi\,\calH_{\mu_x}(x)$. The free Fisher information of $x$ can then be defined as 
$$\Phi(x) := \|\xi\|_{\tau}^2,$$
where  $\|\xi\|_{\tau}^2:=\tau(\xi^2)$. So, the definition of $\Phi_n(\cdot)$ introduced in \Cref{def:finite_free_entropies} can be obtained by replacing the notion of conjugate variable $\xi$ by that of (normalized) score vector $\frac{2}{n-1} \score(\alpha)$, whose entries are the Hilbert transform (up to removing the singularities corresponding to $s=t$) of the root distribution of the polynomial in question.

\paragraph{Acknowledgments.} NS would like to thank Dima Shlyakhtenko for introducing this research direction and for many insightful conversations, and Aurelien Gribinski, Venkat Anantharam, and especially Adam Marcus for helpful discussions. JGV would like to thank Otte Hein\"avaara for stimulating discussions at the early stages of this project and Daniel Perales for comments on an early version of this paper.  We thank IPAM for hosting the workshop ``Free Entropy Theory and Random Matrices'' in February 2025, during which \Cref{thm:monotonicity} was posed as a question during a working group by JGV, and in particular we would like to thank all the members in the program who engaged with the question at the time. NS and ZS were supported by NSF Grant CCF-2420130. 

Finally, we would also like to thank the two anonymous referees for a number of helpful comments and observations. In particular, their reviews motivated us to determine  the equality cases for all of our inequalities, and to flesh out the arguments needed to derive the infinite-dimensional versions of these inequalities by taking $n\to \infty$. 

\paragraph{Bibliographical note.} 
The characterization of the equality cases was not present in the first version of the present paper, which was submitted for review in February, 2026. Theorems \ref{thm:equality_cases_Fisher} and \ref{thm:equality_cases_Lieb_EPI} below have been added in this revised version in July, 2026. These results have also been  obtained recently in independent work  of Hashemi and Hyun in \cite{hashemi2026finite}. Both our approach and the approach of Hashemi and Hyun use the Hermite differential equation (see Lemma \ref{lem:hermite_ode}) to reduce the problem of characterizing the equality cases to the problem of showing that the second singular value of the Jacobian $\jac$ is simple.  The latter statement, which is the core difficulty, is given in Proposition \ref{prop:simple_sigma_derivative} in our paper and Corollary 1.3 in theirs \cite{hashemi2026finite}. Both  proofs of this statement start by applying the Helton-Vinnikov theorem, but differ in later stages. In short, the approach of Hashemi and Hyun seems to use purely real algebraic geometry techniques to conclude the argument, whereas our approach relies on a geometric argument about the curvature of $\partial  K(g_{\lilconv}, (\bone_n, 0_n))$, described in \cref{sec:intro_eq_cases}. 

In addition to characterizing the equality cases, for the revised version of this paper,  we have made some minor corrections, additions, and notational changes. Furthermore, we have at the request of the referees added  \cref{sec:large_n_limit}, where we provide a detailed derivation of the original free probability inequalities proven in \cite{szarek1996volumes, voiculescu1998analogues, shlyakhtenko2022fractional} by taking $n\to \infty$ in our results. As an unexpected byproduct of fleshing out these arguments, we realized that Voiculescu's free Young's inequality for $p=\infty$ can be derived from the preservation of mesh for the finite free convolution (see \cref{sec:free_proba_inputs}), which may be of independent interest.    

\begin{remark}[Alternate proof of Theorem \ref{thm:stams_inequality}] The statement of 
    Theorem \ref{thm:stams_inequality} was included as Problem $\#4$ in the first batch of the \emph{First Proof} project \cite{abouzaid2026first} in early February, 2026. OpenAI's team produced a proof of this theorem that was significantly different from ours during the week of February $6-13$, via a process detailed here \cite{OAI2026}. The authors have read OpenAI's proof and believe it to be correct, though they did not check every line, and it has been formalized in Lean by \cite{archon26}.
\end{remark}

\paragraph{Statement on AI Use.} OpenAI Codex 5.5 was used to write a visualization app for the hyperbolicity cones defined in equation (\ref{eq:cone_boundary_convo}) for $n=3$, while working on the proof of Theorem \ref{thm:equality_cases_Fisher} in July, 2026. Apart from this instance, no AI tools were used in discovering, proving, or writing any of the content in this paper.

\section{Preliminaries}
\label{sec:preliminaries}

As discussed in the introduction, we will be interested on how differential operators transform the roots of the real-rooted polynomials on which they act. To enable this discussion, let us first introduce some notation and terminology. 
\begin{definition}[Vectors of roots]
    Given a real-rooted polynomial $p(x)$ of degree $n$, we will say that $\alpha\in \bR^n$ is a vector of roots for $p(x)$ if its entries, denoted by $\alpha_1, \dots, \alpha_n$,  enlist all the roots of $p(x)$. Moreover:
\begin{enumerate}[i)]
    \item (Ordered vectors)  We will say that $\alpha\in \bR^n$ is ordered if $\alpha_1\ge \cdots \ge \alpha_n$.
    \item (Simple vectors) We will say that $\alpha\in \bR^n$ is simple if all of its entries are distinct.
\end{enumerate}    
\end{definition}

 Conversely, given a vector $\alpha\in \bR^n$, one can build the real-rooted polynomial
$$\pmap{\alpha}(x) := \prod_{i=1}^n(x-\alpha_i).$$
Then, as discussed in \Cref{sec:introduction}, it will be natural to view collections of roots as discrete probability measures so, for $\alpha\in \bR^n$, we define 
$$m_1(\alpha):=\frac{1}{n} \sum_{i=1}^n \alpha_i, \qquad m_2(\alpha):= \frac{1}{n}\sum_{i=1}^n \alpha_i^2, \qquad \text{and} \qquad \var(\alpha):= m_2(\alpha)-m_1(\alpha)^2. $$
Everywhere below, we will use $\bone_n$ to denote the all-ones vector in $\bR^n$.

\subsection{Roots of derivatives} In order to study how differentiation transforms the roots of a polynomial, we  define the following map. 

\begin{definition}[Map for roots of the derivative]
\label{def:derivative_map}
    For every $\alpha\in \bR^n$ let $$\Omega_{\scriptscriptstyle\partial_x}^1(\alpha)\ge  \dots\ge  \Omega_{\scriptscriptstyle\partial_x}^{n-1}(\alpha)$$ be the roots of the polynomial $\partial_x\pmap{\alpha}(x)$, and let 
    $$\dermap(\alpha) := \adj{(\Omega_{\scriptscriptstyle\partial_x}^1(\alpha),  \dots,   \Omega_{\scriptscriptstyle\partial_x}^{n-1}(\alpha))}.$$  
\end{definition}

We remark that, by Rolle's theorem, if $p(x)$ is a real-rooted polynomial, then $p'(x)$ is also real-rooted. So $\dermap(\alpha)$ can be thought of as a map from $\bR^n$ to $\bR^{n-1}$. Moreover, as Rolle's theorem provides that the roots of $p'(x)$ interlace those of $p(x)$, then if the roots of $p(x)$ are simple then the roots of $p'(x)$ are also simple. This yields the following observation. 

\begin{observation}[Smoothness]
\label{obs:smoothness_of_dermap}
    The map $\dermap: \bR^n \to \bR^{n-1}$ is a smooth function in an open neighborhood of any simple vector $\alpha\in \bR^n$. 
\end{observation}

\begin{proof}
Recall that the roots of a real-rooted polynomial $p(x)$ are a smooth function of the coefficients of $p(x)$, provided that $p(x)$ has simple roots. Then, the result follows from the fact that the coefficients of $p'(x)$ are a smooth (in fact polynomial)  function of the coefficients of $p(x)$, and $p'(x)$ has simple roots whenever $p(x)$ has simple roots. 
\end{proof}

It will also be important to understand how $\dermap$ affects the variance of the root distribution. 

\begin{observation}[Comparing moments]
\label{obs:variancescaling}
   For any $\alpha\in \bR^n$ the following hold:
    $$ m_1( \dermap(\alpha)) = m_1(\alpha)\qquad \text{and}\qquad \var(\dermap(\alpha)) = \tfrac{n-2}{n-1}\cdot \var(\alpha). $$
    In particular, if $\alpha^\top \bone_n =0$ then $\|\dermap(\alpha)\|^2 = \frac{n-2}{n}\|\alpha\|^2$.
\end{observation}

\begin{proof}
  Let $p(x) = \pmap{\alpha}(x)$ and  expand 
  $$p(x) = \sum_{k=0}^n (-1)^k a_k x^{n-k}\qquad \text{and}\qquad \frac{1}{n}\partial_x p(x) = \sum_{k=0}^{n-1} (-1)^k b_k x^{n-k-1},$$ where $a_0=1$. Then, by alternatively writing $\frac{1}{n}\partial_x p(x) = \frac{1}{n}\sum_{k=0}^n (n-k)a_k x^{n-k-1}$, and equating coefficients we get
\begin{equation}
\label{eq:coefrelations}
b_0=a_0=1, \qquad b_1= \tfrac{n-1}{n}a_1, \qquad \text{and} \qquad b_2 = \tfrac{n-2}{n}a_2.
\end{equation}
  On the other hand, Vieta's relations give $m_1(\alpha) = \frac{1}{n}a_1$ and $m_1(\dermap(\alpha))=\frac{1}{n-1}b_1$, which combined with (\ref{eq:coefrelations}) yields $m_1(\dermap(\alpha)) = m_1(\alpha)$. 
    
    Second, since $\partial_x (p(x-c)) = (\partial_xp)(x-c)$ for any $c\in \bR$, and shifts do not change the variance of a distribution, to prove the claim about variances we can  assume, without loss of generality, that $m_1(\alpha)=m_1(\dermap(\alpha))=0$ and therefore $a_1=b_1=0$. Now use the Vieta relations one more time and the assumption that $a_1=b_1=0$ to get that 
\begin{equation}
\label{eq:vietarelations}
m_2(\alpha) = \frac{1}{n}( a_1^2-2a_2) = -\frac{2a_2}{n} \qquad \text{and}\qquad m_2(\dermap(\alpha)) =\frac{1}{n-1}(b_1^2-2b_2) = -\frac{2b_2}{n-1}.
\end{equation}
Now, because $m_1(\alpha)=m_1(\dermap(\alpha))=0$ we have that $\var(\alpha)=m_2(\alpha)$ and $\var(\dermap(\alpha))=m_2(\dermap(\alpha))$. So, the claim about $m_1$ and $\var$ is obtained by putting (\ref{eq:coefrelations}) and (\ref{eq:vietarelations}) together. Finally, the last claim follows trivially from the same computations. 
\end{proof}

\subsection{The heat flow and Hermite polynomials} Other than the operator $\partial_x$, the (reverse) heat flow operator $\heat$ plays a distinguished role in this context. This operator drives the evolution, with respect to $t$ (time), of solutions to the standard heat equation. Precisely, the function $u(x, -t) := \heat p(x)$ solves 
$$\partial_t u(x, t) = \tfrac{1}{2} \partial_x^2 u(x, t),$$
with initial condition $u(x, 0) = p(x)$. Importantly, if $p(x)$ is a polynomial of degree $n$, then $p_t(x) := \heat p(x)$ will also be a polynomial of degree $n$ for every $t$. Moreover, it is a classical fact that when the initial condition is given by $p(x)=x^n$ then 
$$\heat x^n = \sqrt{t}_*H_n(x),$$
where $H_n(x)$ denotes the $n$-th  (probabilist) monic Hermite polynomial (e.g. see \cite[\S 2.1]{hall2025zeros}). It follows that one can implement the heat flow by taking finite free convolutions with Hermite polynomials of different scales. Namely, if $p(x)$ is a polynomial of degree $n$ and we write $p(x) = \hat{p}(\partial_x) x^n$ for $\hat{p}$ a polynomial, then 
\begin{equation}
    \heat p(x) = \heat \hat{p}(\partial_x) x^n = p(x)\boxplus_n \sqrt{t}_*H_n(x).\label{eq:heat flow of poly}
\end{equation}
Hence, since $\boxplus_n$ preserves real-rootedness, if $p(x)$ is a real-rooted polynomial then so is $\heat p(x)$ for all $t\ge0$.

In several instances throughout the paper we will crucially use the following description (see \cite[Lemma 2.4]{csordas1994lehmer} for a proof) of the root dynamics of a polynomial under the heat flow, which will allow us to express the score vector of a  vector of roots as a derivative.   

\begin{lemma}[Score vectors as derivatives]
\label{lem:score_and_derivatives}
    Let $p(x)$ be a real-rooted polynomial and $\alpha\in \bR^n$ be its ordered vector of roots. For $t\in \bR$ let $p_t(x):= \heat p(x)$ and $\alpha(t)\in \bR^n$ be the ordered vector of roots of $p_t(x)$. Then, $\alpha(t)$ is differentiable at every $t$ for which $p_t(x)$ has simple roots, in which case it holds that
     $$\alpha_i'(t) = \sum_{j : j\neq i} \frac{1}{\alpha_i(t)-\alpha_j(t)},\quad \text{for all $i=1, \dots, n$}.$$
 In particular, if $p(x)$ has simple roots, we have $\alpha'(0) = \score(\alpha)$. 
\end{lemma}

We now record an observation that will be used repeatedly in our proofs. This is a direct consequence of the definition of $\score(\alpha)$. 

\begin{observation}[Orthogonality]
\label{obs:orthogonality_score}
    For every $\alpha\in \bR^n$,  $\score(\alpha)$ is orthogonal to $\bone_n$. 
\end{observation}

Finally, to establish that Hermite polynomials are the only polynomials that attain equality in our inequalities, we will need the following well-known ODE description. 

\begin{lemma}[Hermite differential equation]
\label{lem:hermite_ode}
   If $f(x)$ is a degree $n$ polynomial satisfying the differential equation
    $$cf''(x)-xf'(x)+n f(x)=0,$$
    then $f(x) = a H_n(bx)$, where $H_n(x)$ is the degree $n$ Hermite polynomial and $a, b\in \bR$. 
\end{lemma}

\subsection{Hyperbolic polynomials}
\label{sec:geometry_of_polynomials}

 In what follows, we will use $\bR[x_1, \dots, x_m]$ to denote the set of multivariate polynomials with real coefficients in  $m$ variables. And, we will say that $f\in \bR[x_1, \dots, x_m]$ is a homogeneous polynomial if all of its monomials have the same total degree. In particular, we will work with the class of hyperbolic polynomials, which is a natural multivariate generalization of the class of real-rooted univariate polynomials.

\begin{definition}[Hyperbolic polynomial]
    Let $f\in \bR[x_1, \dots, x_m]$ be homogeneous and let $u\in \bR^m$. We say that $f$ is hyperbolic in the direction $u$ if $f(u)>0$ and the univariate polynomial $g_\zeta(t):=f(\zeta+tu)$ is real-rooted for every $\zeta\in \bR^m$. 
\end{definition}

In \cite[Corollary 3.3]{bauschke2001hyperbolic} Bauschke, G\"uler, Lewis, and Sendov obtained the following generalization of a classical result of G$\overset{\circ}{\text{a}}$rding \cite{gaarding1951linear}.

\begin{theorem}[Bauschke et al.]
\label{thm:baushke}
    Let $f\in \bR[x_1, \dots, x_m]$ be a hyperbolic polynomial in the direction $u\in \bR^m$ and for every $\zeta\in \bR^m$ let $\lambda_1(\zeta)\ge  \dots\ge  \lambda_m(\zeta)$ be the roots of $g_\zeta(t):=f(\zeta+tu)$.  Then, for every $k=1, \dots, m$, the function $$\sigma_k(\zeta) := \sum_{i=1}^k \lambda_i(\zeta)$$ 
    is convex in $\zeta$. 
\end{theorem}

For our proofs we will need the following corollary of the above result. 

\begin{corollary}[Bauschke et al.]
\label{cor:bauschke}
Using the same notation as in \Cref{thm:baushke}, if $c_1\ge \cdots \ge c_m\ge 0$, then
    the function 
    $$\sum_{i=1}^m c_i \lambda_i(\zeta)$$
    is convex in $\zeta$. 
\end{corollary}

\begin{proof}
    Write $\sum_{i=1}^m c_i \lambda_i(\zeta) = c_m \sigma_m(\zeta)+ \sum_{i=1}^{m-1} (c_i-c_{i+1}) \sigma_i(\zeta)$. Then, by \Cref{thm:baushke},  the right-hand side is a positive linear combination of convex functions and therefore it is convex. 
\end{proof}

To determine the equality cases for our inequalities we will make use of a celebrated result of Helton and Vinnikov \cite{helton2007linear}, which states that hyperpbolic polynomials in three variables admit a determinantal representation. 

\begin{theorem}[Helton--Vinnikov]
\label{thm:Helton-Vinnikov}
    Let $f\in \bR[x, s, t]$ be a degree $n$ hyperbolic polynomial  in the direction $(1, 0, 0)$. Then, there exist $n\times n$ Hermitian matrices $A$ and $B$ such that
    $$f(x, s, t) = \det(x-sA-tB). $$
\end{theorem}

Finally, we will also use the well-known notion of hyperbolicity cone. 

\begin{definition}
\label{def:hyperbolicity_cone}
    If $f\in \bR[x_1, \dots, x_m]$ is hyperbolic in the direction $u\in \bR^m$, then the hyperbolicity cone of $f$ with respect to $u$ is the set 
    $$K(f, u) = \{\zeta\in \bR^n : g_\zeta(t)=f(\zeta-tu) \text{ has only positive roots} \}.$$
\end{definition}

\subsection{The finite free convolution} In the sequel, and in connection to hyperbolicity, it will prove useful to have the following alternative description of the finite free convolution (see \cite[Theorem 2.11]{marcus2022finite}). If $p(x)$ and $q(x)$ are polynomials of degree $n$, their finite free convolution is given by
\begin{equation}
\label{eq:finite_free_convo}
p(x)\boxplus_n q(x) = \frac{1}{n!} \sum_{\pi\in S_n} \prod_{i=1}^n(x-\alpha_i-\beta_{\pi(i)}),
\end{equation}
where $\alpha$ and $\beta$ are vectors of roots for $p(x)$ and $q(x)$, respectively, and $S_n$ denotes the symmetric group on $n$ elements. 

As in the derivative case, we will define a map that takes the roots of $p(x)$ and $q(x)$ and outputs the roots of the convolution $p(x)\boxplus_nq(x)$. 

\begin{definition}[Map for roots of the convolution] For every $\alpha, \beta \in \bR^n$ we will use
$$\Omega_{\scriptscriptstyle\boxplus_n}^{1}(\alpha, \beta)\ge \cdots \ge \Omega_{\scriptscriptstyle\boxplus_n}^{n}(\alpha, \beta) $$
    to denote the roots of the polynomial $\pmap{\alpha}(x)\boxplus_n \pmap{\beta}(x)$, and define $$\roots(\alpha, \beta):=\adj{(\Omega_{\scriptscriptstyle\boxplus_n}^{1}(\alpha, \beta), \dots, \Omega_{\scriptscriptstyle\boxplus_n}^{n}(\alpha, \beta))}.$$ 
\end{definition}

As mentioned before,  Walsh \cite{walsh1922location} proved that if $p(x)$ and $q(x)$ are real-rooted polynomials, then so it $p(x)\boxplus_n q(x)$. Hence, we can view $\roots$ as map from $\bR^n\times \bR^n$ to $\bR^n$. Moreover, as with the differential operator $\partial_x$, in this general setting it also holds that if either $p(x)$ or $q(x)$ have simple roots, then $p(x)\boxplus_nq(x)$ is guaranteed to have simple roots. This is a corollary of the following stronger (well-known) result mentioned in the introduction. 

\begin{theorem}[Preservation of mesh]
\label{thm:preservation_of_mesh}
    Let $\alpha, \beta\in \bR^n$.  Then, the mesh of $\pmap{\alpha}\boxplus_n \pmap{\beta}$ the mesh is greater or equal than that of $\pmap{\alpha}$ and $\pmap{\beta}$. In other words
    $$\min_{i\neq j} |\Omega_{\scriptscriptstyle\boxplus_n}^{i}(\alpha, \beta)-\Omega_{\scriptscriptstyle\boxplus_n}^{j}(\alpha, \beta)| \geq \max\{\min_{i\neq j} |\alpha_i-\alpha_j|, \min_{i\neq j} |\beta_i-\beta_j|\}.  $$
\end{theorem}

\begin{proof}
    This is easily proved by combining Riesz's theorem \cite{stoyanoff1925theoreme} with the Hermite-Poulain theorem. We refer the reader to \cite{katkova2011multiplier} for details. 
\end{proof}

We can then establish the following. 

\begin{observation}[Smoothness]
\label{obs:smoothness_convolution}
    Let $\alpha, \beta\in \bR^n$. If $\alpha$ is a simple vector, then $\roots$ is differentiable in an open neighborhood of $(\alpha, \beta)$. The same conclusion holds if $\beta$ is a simple vector.
\end{observation}

\begin{proof}
As in the proof of \Cref{obs:smoothness_of_dermap} we use that, for polynomials with simple roots, the roots are a differentiable function of the coefficients. Then, since the coefficients of $\pmap{\alpha}(x)\boxplus_n \pmap{\beta}(x)$ are a polynomial function of $\alpha$ and $\beta$, and  $\pmap{\alpha}(x)\boxplus_n \pmap{\beta}(x)$ has simple roots if $\alpha$ is simple, the result follows.    
\end{proof}

    %It is worth noting that in the proof of \Cref{obs:smoothness_convolution}, the root simplicity result implied by Fujie's theorem is only being used to locate an explicit  region of differentiability for $\Omega_{\scriptscriptstyle\boxplus_n}$. However,  for the use cases in this paper, it will be enough to know that the set of points of differentiability of $\roots$ is dense in $\bR^n\times \bR^n$, and this can  alternatively be obtained by an elementary perturbative argument. 

Finally, we state some  well-known properties of the finite free convolution that will be used in the sequel. 

\begin{proposition}[Properties of $\boxplus_n$]
\label{prop:properties_of_ff_convolutions}
    For all $\alpha, \beta\in \bR^n$ the following hold:
    \begin{enumerate}[(i)]
        \item \label{item:translations} (Commutation with translation)   For all $t\in \bR$ it holds that 
        $$\roots(\alpha+t\bone_n ,\beta) = \roots(\alpha, \beta) +t\bone_n \qquad \text{and}\qquad  \roots(\alpha , \beta + t \bone_n) = \roots(\alpha, \beta)+t\bone_n.$$  
                \item \label{item:additivity} (Additivity) The first two moments satisfy 
        $$m_1(\roots(\alpha, \beta))=m_1(\alpha)+m_1(\beta)\qquad \text{and}\qquad \var(\roots(\alpha, \beta)) = \var(\alpha)+\var(\beta).$$ 
        \item \label{item:interlacing} (Interlacing preservation) If $\alpha, \widetilde{\alpha}\in \bR^n$ and $\alpha$ interlaces $\widetilde{\alpha}$, by which we mean $\alpha_1\ge \widetilde{\alpha}_1\ge \alpha_2 \cdots\ge \alpha_n \ge \widetilde{\alpha}_n$,  then $\roots(\alpha, \beta)$ interlaces $\roots(\widetilde{\alpha},\beta)$. 
    \end{enumerate}
\end{proposition}

\begin{proof}
    First note that (\ref{item:translations}) follows directly from \eqref{eq:finite_free_convo}. The claims in (\ref{item:additivity}) can also be proven directly from \eqref{eq:finite_free_convo} via Vieta's relations, but also follow from the much more general theory of finite free cumulants \cite{arizmendi2018cumulants}. Finally, (\ref{item:interlacing}) follows from the fact that  $\boxplus_n$ preserves real-rootedness and Theorem 1.43 in \cite{fisk2006polynomials} (also see \cite[Proposition 2.11]{martinez2024real} for an exposition in the context of finite free probability). 
\end{proof}

\section{Relating score vectors}
\label{sec:blachman}

Our proofs of Theorems \ref{thm:fisher_monotonicity} and \ref{thm:stams_inequality} are inspired by Blachman's  proof of Stam's inequality \cite{blachman2003convolution}, which we review below to provide some motivation. A crucial notion in Blachman's argument, as well as in the existing proofs of the free Stam inequality \cite{shlyakhtenko2007free, shlyakhtenko2022fractional}, is that  of conditional expectation and therefore that of joint distribution.  However, as  convolutions in finite free probability are not produced by random variables interacting with each other, in this context there is no notion of joint distribution and therefore no clear notion of conditional expectation. In what follows we posit that, in the context of entropy inequalities, the Jacobians of the root maps $\dermap$ and $\roots$ can serve as a replacement for conditional expectations, and we successfully exploit this perspective in the proofs of Theorems \ref{thm:fisher_monotonicity} and \ref{thm:stams_inequality}.   

\subsection{Blachman's proof} 

To discuss the proof of Blachman \cite{blachman2003convolution} we begin by recalling the definitions from the classical setting, where we will only consider absolutely continuous random variables and, for $X$  a  random variable, we will use $\rox(x)$ to denote its probability density function. 

The (classical) score function of a random variable $X$ is defined as $\frac{\rho_X'(x)}{\rox(x)}$ and its (classical) Fisher information as $$\calI(X) := \E\lbr{\lpr{\tfrac{\rox'(X)}{\rox(X)} }^2}.$$

Blachman's key observation is that the score function of a sum of independent random variables can be related, via conditional expectations, to the score functions of the original random variables. 

\begin{lemma}[Blachman]
\label{lem:blachman}
     Let $X$ and $Y$ be independent random variables and $Z:= X+Y$.  Then, for any $z\in \bR$ one has
 \begin{equation}
 \label{eq:blachman_score_functions}
    \frac{\roz'(z)}{\roz(z)} = \E\left[ \left. \frac{\rox'(X)}{\rox(X)} \right| Z=z \right] = \E\left[\left. \frac{\roy'(Y)}{\roy(Y)}\right| Z=z\right]. 
 \end{equation}
\end{lemma}
\iffalse
\begin{proof}
     By definition  $\roz(z) = \int \rox(z-x) \roy(x) dx$.  Differentiating under the integral yields
    $$\roz'(z) = \int \rox'(z-x) \roy(x) dx = \int \rox'(x) \roy(z-x) \d x.$$
   Therefore, we have  
    \begin{equation}
    \label{eq:manipulating_densities}
    \frac{\roz'(z)}{\roz(z)} = \int \frac{\rox(x)\roy(z-x)}{\roz(z)} \frac{\rox'(x)}{\rox(x)} \d x.
      \end{equation}
    Now, since  $\frac{\rox(x)\roy(z-x)}{\roz(z)}$ is precisely the density of $X$ conditioned on $Z=z$, it follows that the right-hand side of (\ref{eq:manipulating_densities}) equals $ \E\left[ \left. \frac{\rox'(X)}{\rox(X)} \right| Z=z \right] $. The proof that  $\frac{\roz'(z)}{\roz(z)}  = \E\left[\left. \frac{\roy'(Y)}{\roy(Y)}\right| Z=z\right]$ is completely analogous. 
\end{proof}
\fi
With \Cref{lem:blachman} at hand the classical Stam inequality can be derived readily. Indeed, using the notation from \Cref{lem:blachman}, Blachman's identity implies that for any $a, b\in \bR$
 \begin{align*}
        (a+b)^2 \left( \tfrac{\roz'(z)}{\roz(z)} \right)^2 &= \E \left[ \left.a\,\tfrac{\rox'(X)}{\rox(X)} + b\,\tfrac{\roy'(Y)}{\roy(Y)}\right| Z=z \right]^2
        \\ & \le \E \left[ \left(\left.a\, \tfrac{\rox'(X)}{\rox(X)} + b\, \tfrac{\roy'(Y)}{\roy(Y)}\right)^2\right| Z=z \right],
    \end{align*}
    where the last inequality follows from Jensen's inequality.  Then, integrating over $z$ and expanding the square we get
        \begin{equation}
        \label{eq:expanding_squares}
        (a+b)^2 \E\Big[ \big( \tfrac{\roz'(Z)}{\roy(Z)} \big)^2 \Big] \le a^2 \E\Big[ \big(\tfrac{\rox'(X)}{\rox(X)} \big)^2 \Big] + b^2 \E\Big[ \big(\tfrac{\roy'(Y)}{\roy(Y)}\big)^2 \Big]+2ab \E\Big[ \tfrac{\rox'(X)}{\rox(X)}\Big] \E\Big[\tfrac{\roy'(Y)}{\roy(Y)}\Big].
            \end{equation}
        Finally, observe that $\E\Big[ \tfrac{\rox'(X)}{\rox(X)}\Big]  = \int \rox'(x) dx =0,$
        since for $\rox$ to be a continuous probability density, it must  decay to 0 at $\pm \infty$. This yields 
        $$(a+b)^2 \calI(Z)  \le a^2 \calI(X) + b^2 \calI(Y),$$
        and one can get the classical Stam inequality by substituting $a= \frac{1}{\calI(X)}$ and $b=\frac{1}{\calI(Y)}$. 
 
\subsection{Jacobians of root maps}
\label{sec:finite_scores_to_finite_scores}

In this section, to derive the analogue of \Cref{lem:blachman} we will exploit the root dynamics under the heat flow described in \Cref{lem:score_and_derivatives}. Importantly, this approach bypasses the need for the notion of joint distribution  crucially used in \Cref{lem:blachman}. 

In what follows we will fix $\alpha, \beta\in \bR^n$ and assume that $\alpha$ and $\beta$ are ordered and simple. We will then use $\jac$ to denote the Jacobian of the map $\roots$ at $(\alpha, \beta)$, which we will view as a linear transformation
$$\jac: \bR^n \oplus \bR^n \to \bR^n.$$

\begin{lemma}[Relating score vectors for $\roots$]
\label{lem:score_to_score_convolution}
For all $a, b \in \bR$ it holds that 
$$  \jac \big(a\score(\alpha)\oplus b \score(\beta)\big)=(a+b)\score(\roots(\alpha, \beta))  . $$
\end{lemma}

\begin{proof} Let $\gamma:=\roots(\alpha, \beta)$ and set $p(x) := \pmap{\alpha}(x), q(x) := \pmap{\beta}(x),$ and $r(x) = \pmap{\gamma}(x)$. Then, for every $t\in \bR$ define $p_t(x) := \heat p(x)$, $q_t(x):= \heat q(x)$, and $r_t(x):=\heat r(x)$, and let $\alpha(t), \beta(t),$ and $\gamma(t)$ be the ordered vectors of roots for $p_t(x), q_t(x)$ and $r_t(x)$, respectively. Since the finite free convolution commutes with  differential operators,  it follows that
      \begin{align*}
     p_{at}(x) \boxplus_n q_{bt}(x) &= \left(e^{-\frac{at}{2} \partial_x^2}p(x)\right) \boxplus_n \left(e^{-\frac{bt}{2} \partial_x^2}q(x)\right)
      \\ & = e^{-\frac{(a+b)t}{2} \partial_x^2} \left(p(x) \boxplus_n q(x)\right)
      \\ & = r_{(a+b)t}(x). 
            \end{align*}
    Hence  $$  \roots(\alpha(at), \beta(bt))= \gamma((a+b)t), \quad \text{for all $t\ge 0.$}$$ 
    Differentiating the above  relation with respect to $t$ and using the chain rule for the left-hand side,  we get that
      \begin{equation}
      \label{eq:intermediate_relation}
      \jac \left(
           a \, \alpha'(0)\oplus    b\, \beta'(0) \right)= (a+b)\gamma'(0) .
                 \end{equation}
   Finally, note that \Cref{lem:score_and_derivatives} yields $\alpha'(0) =\score(\alpha), \beta'(0) = \score(\beta)$, and $\gamma'(0)= \score(\gamma)$. Then, the proof is concluded by plugging this information back into (\ref{eq:intermediate_relation}). 
\end{proof}

Similarly, we will use $\jacder$ to denote the Jacobian of $\dermap$ at $\alpha$ (which as mentioned above has distinct entries by assumption) and   view it as a linear transformation
$$\jacder : \bR^n \to \bR^{n-1}.$$
\begin{lemma}[Relating score vectors for $\dermap$]
\label{lem:score_to_score_der}
    The following holds
    $$\scoreder(\dermap(\alpha))=\jacder \big(\score(\alpha)\big) .$$
\end{lemma}

\begin{proof}
    Let $p(x) := \pmap{\alpha}(x)$  and for $t\in \bR$ define the heat flow evolutions $p_t(x) := \heat p(x)$ and $(p')_t(x) := \heat p'(x) $. Let $\delta(t)$ be the ordered vector of roots for $(p')_t$.
    
    Because $\heat$ commutes with $\partial_x$, we have that $(p_t)'= (p')_t$ and therefore 
    $$ \delta(t)= \dermap(\alpha(t))  .$$
    So, differentiating both sides, and applying the chain rule to the right-hand side, we get $ \delta'(0)= \jacder (\alpha'(0)).$ On the other hand, \Cref{lem:score_and_derivatives} implies that $\alpha'(0) = \score(\alpha)$ and $\delta'(0) = \scoreder(\dermap(\alpha))$. 
\end{proof}

\subsection{Double stochasticity}
\label{sec:double_stochasticity}
In \Cref{lem:score_to_score_convolution} the Jacobian $\jac$ maps the score vectors of $\alpha$ and $\beta$ to the score vector of $\roots(\alpha, \beta)$, in analogy to \Cref{lem:blachman},  where the conditional expectation $\E[\,\cdot\, | Z=z]$ maps the score functions of $X$ and $Y$ to the score function of $Z$. Here we show that this analogy is valid in a deeper sense. Concretely, we show that the Jacobian $\jac(\cdot, \cdot)$ can be decomposed into two doubly stochastic matrices (which may be interpreted as couplings of probability distributions), one that relates $\alpha$ and $\gamma$, and another one that relates $\beta$ and $\gamma$. This result is included here, as we think it is interesting in its own right, but it is worth noting that the hyperbolicity approach,  presented in \Cref{sec:stam_geometry_proof},  proves Theorems \ref{thm:fisher_monotonicity} and \ref{thm:stams_inequality} without using \Cref{lem:double_stochasticity_convolution}, \Cref{lem:jacder doub stoch}, or \Cref{obs:norm_bound_jacobian}.

Interestingly, the double stochasticity is a consequence of the properties of $\boxplus_n$ enlisted in \Cref{prop:properties_of_ff_convolutions}. 

\begin{lemma}[Double stochasticity for $\boxplus_n$]
\label{lem:double_stochasticity_convolution} 
    Let $\palpha$  be the $n\times n$ matrix associated to the Jacobian of the map $\roots(\cdot, \beta) : \bR^n \to \bR^n$ at $\alpha$, and let $\pbeta$ be the $n\times n$ matrix associated to the Jacobian of the map $\roots(\alpha,\cdot ): \bR^n \to \bR^n$ at $\beta$. Then:
    \begin{enumerate}[(i)] 
    
    \item \label{item:row_sums} (Row sums) $\palpha \bone_n = \bone_n = \pbeta \bone_n$. 
    \item \label{item:column_sums} (Column sums) $\adj{\bone}_n \palpha = \adj{\bone}_n = \adj{\bone}_n \pbeta$. 
        \item \label{item:positive_entries} (Nonnegative entries) $\palpha$ and $\pbeta$ have nonnegative entries. 

    \end{enumerate}
\end{lemma}

\begin{proof}
    We will just prove the claims for $\palpha$, as the proofs for  $\pbeta$ go in the same way. To prove (\ref{item:row_sums}) we invoke \Cref{prop:properties_of_ff_convolutions} (\ref{item:translations}), which states that
    $$\roots(\alpha+t\bone_n, \beta) = \roots(\alpha, \beta) + t\bone_n.$$
Then, in the above identity, view $\beta$ as fixed and differentiate both sides with respect to $t$. For the left-hand side, the chain rule yields $\palpha \bone_n$, while for the right-hand side we get $\bone_n$, proving the claim. 

To prove (\ref{item:column_sums}) we will show that $\adj{\bone}_n \palpha u = \adj{\bone}_n u $ for all $u\in \bR^n$. To this end fix $u\in \bR^n$ and invoke the first moment identity from \Cref{prop:properties_of_ff_convolutions} (\ref{item:additivity}), which states that 
$$\adj{\bone}_n \roots(\alpha+tu, \beta) =  \adj{\bone}_n \alpha + \adj{\bone}_n\beta + t\adj{\bone}_nu.$$
    Again, we differentiate both sides of the above equation. For the left-hand side the chain rule yields $\adj{\bone}_n \palpha u$ whereas the right-hand side yields $\adj{\bone}_n u$.  

    Finally, to prove (\ref{item:positive_entries}) we will show that any two entries in the same column of $\palpha$ have the same sign, which, since we know that $\adj{\bone}_n \palpha  = \adj{\bone}_n$, will imply that all the entries of $\palpha$ are nonnegative. Proceeding by contradiction, assume that there are $1\le  j, k, \ell \le n$ such that $\palpha_{k, j}$ and $\palpha_{\ell, j}$ have different  signs. Then, use that
    $$\palpha_{k, j} = \frac{\partial \Omega_{\scriptscriptstyle\boxplus_n}^{k}(\alpha, \beta)}{\partial \alpha_j} \qquad \text{and} \qquad \palpha_{\ell, j} = \frac{\partial \Omega_{\scriptscriptstyle \boxplus_n}^{\ell}(\alpha, \beta)}{\partial \alpha_j},$$
where we are viewing $\beta$ as fixed. Now, the assumption that the above have different signs implies that, if we set $\gamma(t):= \roots(\alpha + t e_j, \beta)$, then for all  small enough times $t>0$ we will either have $\gamma_k(t) < \gamma_k(0)$ and $\gamma_l(t) > \gamma_l(0)$, or $\gamma_k(t) > \gamma_k(0)$ and $\gamma_l(t) < \gamma_l(0)$. In either case it will be impossible for $\gamma(t)$ to interlace $\gamma(0)$. This is a contradiction, since for  small enough $t$, $\alpha+te_j$ interlaces $\alpha$ and by \Cref{prop:properties_of_ff_convolutions} (\ref{item:interlacing}) we know that  $\boxplus_n$ is an interlacing preserver. 
\end{proof}

The analogous statement for $\dermap$ can be proven in the same way. However, in this case the matrix in question is  $(n-1)\times n$  and the column sums equal $\frac{n-1}{n}$ instead 1. 
 
\begin{lemma}[Double stochasticity for $\partial_x$]\label{lem:jacder doub stoch} Let $\pdelta$ be the $(n-1)\times n$ matrix associated to the Jacobian $\jacder$. Then:
\begin{enumerate}[(i)]
    \item \label{item:rowsums_der} (Row sums) $\pdelta\bone_n=\bone_{n-1}$.
    \item \label{item:colsums_der} (Column sums) $\adj{\bone}_{n-1}\pdelta=\frac{n-1}{n}\adj{\bone}_{n}$. 
    \item \label{item:non-neg_entries_der} (Nonnegative entries) The entries of $\pdelta$ are nonnegative.
\end{enumerate}
\end{lemma}

\begin{proof}
    For (\ref{item:rowsums_der}), again we note that $\dermap(\ga+t\bone_n)=\dermap(\ga)+t\bone_{n-1}$ and differentiate both sides with respect to $t$. For (\ref{item:colsums_der}), we use the first moment identity from \Cref{obs:variancescaling} which, for every $u\in \bR^n$,  yields  $m_1\lpr{\dermap(\ga+tu)}=m_1(\ga+tu)$. We then rewrite this as
    \[\tfrac{1}{n-1}\adj{\bone}_{n-1}\dermap(\ga+tu)=\tfrac{1}{n}\adj{\bone}_n(\ga+tu).\]
   Then, differentiating with respect to $t$ at $t=0$ and rearranging factors of $n$ yields $$\adj{\bone}_{n-1}\pdelta u=\tfrac{n-1}{n}\adj{\bone}_n u.$$ 
   The claim then follows from the fact that the above holds for every $u\in \bR^n$. 

    For (\ref{item:non-neg_entries_der}), it suffices that in each column, every entry has the same sign. As by (\ref{item:colsums_der}) the column sum is positive, the claim will follow. We proceed by contradiction with fixed column $i$: suppose there are indices $j$ and $k$ such that, without loss of generality, $\pdelta_{j,i}>0>\pdelta_{k,i}$. Thus, at sufficiently small $t$, we will have $\Omega_{\scriptscriptstyle\dd_x}^{j}(\ga+te_i)>\Omega_{\scriptscriptstyle\dd_x}^{j}(\ga)$ and $\Omega_{\scriptscriptstyle\dd_x}^{k}(\ga+te_i)<\Omega_{\scriptscriptstyle\dd_x}^{k}(\ga)$; thus $\dermap(\ga+te_i)$ will not interlace $\dermap(\ga)$, even though $\ga+te_i$ interlaces $\ga$ and $\dd_x$ preserves interlacing (e.g. see \cite[Theorem 1.47]{fisk2006polynomials}). This is the desired contradiction. 
\end{proof}

\subsection{Further remarks}

In view of Lemmas \ref{lem:score_to_score_convolution} and \ref{lem:double_stochasticity_convolution}, for any $\alpha, \beta\in \bR^n$ with distinct entries and every $a, b\in \bR$, we can write 
\begin{equation*}
(a+b)\score(\roots(\alpha, \beta)) = a \palpha \score(\alpha)+ b\pbeta \score(\beta),
\end{equation*}
where $\palpha$ and $\pbeta$ are doubly stochastic matrices. Now, since doubly stochastic matrices satisfy $\|\palpha\| = \|\pbeta\|=1$,  the triangle inequality yields
\begin{equation}
\label{eq:norms_inequality}
\lvert a+b\rvert\|\score(\roots(\alpha, \beta)) \|\le \lvert a\rvert\| \score(\alpha)\|+\lvert b\rvert \| \score(\beta) \|.
\end{equation}
Note that the above already provides the non-trivial inequality $\Phi_n(p\boxplus_n q) \le \Phi_n(p)$ for any two real-rooted polynomials $p(x)$ and $q(x)$. However, to get the full strength of \Cref{thm:stams_inequality} one must exploit the interactions between $\palpha \score(\alpha)$ and $\pbeta \score(\beta)$, in the same way in which the independence between the random variables $X$ and $Y$ was exploited in Blachman's argument when obtaining (\ref{eq:expanding_squares}). 

As far as we can see, to capture the aforementioned interactions, it is not possible to treat $\palpha$ and $\pbeta$ separately, and one must instead understand the Jacobian $\jac$ as a whole. This can be articulated as follows. 
\begin{observation}[Norm of Jacobian]
\label{obs:norm_bound_jacobian}
    $\|\jac\|=\sqrt{2}$ and $\bone_n$ and $\bone_n\oplus \bone_n$ are, respectively, the left and right top singular vectors of $\jac$.  
\end{observation}

\begin{proof}
    As mentioned before, for any $u, v\in \bR^n$ and $a, b \in \bR$ we have $\|\jac(au\oplus b v)\| \le  \lvert a\rvert\|u\|+\lvert b\rvert\|v\|$. So, taking $a=b=1$, and applying the  AM-QM inequality we get that 
    \begin{equation}
    \label{eq:norm_bound}
    \|\jac(u\oplus v) \| \le \sqrt{2(\|u\|^2+\|v\|^2)} = \sqrt{2}\|u\oplus v\|,
        \end{equation}
    and therefore $\|\jac  \| \le \sqrt{2}$. 

    The claim about the top singular vectors follows from \Cref{lem:double_stochasticity_convolution} which implies that $\jac(\bone_n\oplus \bone_n) = 2\bone_n$ and $\jac^*(\bone_n) = \bone_n\oplus \bone_n$. In turn, we see that $\bone_n$ and $\bone_n\oplus \bone_n$ achieve equality in the norm bound give in (\ref{eq:norm_bound}). 
\end{proof}

However, \Cref{obs:norm_bound_jacobian} does not suffice to prove \Cref{thm:stams_inequality}. For this, we will need a bound on the second largest singular value of $\jac$. We do this in \Cref{sec:stam_geometry_proof}. 

Similarly, in the case of $\partial_x$, \Cref{lem:jacder doub stoch} implies that 
$$\|\jacder\| = \sqrt{\tfrac{n-1}{n}},$$
and that $\bone_n$ and $\bone_{n-1}$ are, respectively, the right and left singular vectors of $\jacder$. However, to prove \Cref{thm:fisher_monotonicity}, we will need a bound on the second largest singular value of $\jacder$, which we obtain below in \Cref{sec:matrix_analysis_fisher_monotonicity}.

\section{Matrix-analytic proof of Fisher information monotonicity}
\label{sec:matrix_analysis_fisher_monotonicity}

Here we provide our first proof of \Cref{thm:fisher_monotonicity}. In \Cref{sec:stam_geometry_proof} we will provide an alternative proof which proceeds by ideas from the geometry of polynomials. 

Throughout this section we will fix a simple ordered vector $\alpha\in \bR^n$, use $\jacder$ to denote the Jacobian of $\dermap$ at $\alpha$ and, as in \Cref{lem:jacder doub stoch}, use $\pdelta$ to denote the associated $(n-1)\times n$ matrix. Our proof starts by observing that the entries of $\pdelta$ admit an explicit formula in terms of the $\alpha_i$ and the $\dermapi(\alpha)$. Moreover, these entries are, in fact, the weights from the Gauss--Lucas theorem. Then, we utilize  elementary facts from matrix analysis to upper bound  the second singular value of $\jacder$. 

\subsection{The Gauss--Lucas matrix}

Let $p(z):= \prod_{i=1}^n (z-\omega_i)$ where the $\omega_i$ are allowed to be complex and let $\zeta_1, \dots, \zeta_{n-1}$ be the roots of $p'(z)$. Then, the classic Gauss--Lucas theorem states that each $\zeta_i$ can be written as a convex combination of $\omega_1, \dots, \omega_n$. Moreover, the standard proof of Gauss--Lucas yields that for every $i$ one has
\begin{equation}
\label{eq:Gauss--Lucas_theorem}
\zeta_i :=  \sum_{j=1}^n \frac{1}{Z_i\lvert \zeta_i-\omega_j\rvert^2}\cdot \omega_j
\end{equation}
where $Z_i := \sum_{j=1}^n \frac{1}{\lvert \zeta_i-\omega_j\rvert^2}$. It turns out that, for our setting,  these convex weights are precisely the entries of the  matrix $\pdelta$.

\begin{lemma}[$\pdelta$ is the Gauss--Lucas matrix]
\label{lem:gauss_lucas_entries}
    Assume $\alpha\in \bR^n$ is a simple vector. Then, for any fixed $i=1, \dots, n-1$ and $j=1, \dots, n$ one has  
    $$\pdelta_{ij} = \frac{1}{Z_i}\cdot \frac{1}{(\dermapi(\alpha)-\alpha_j)^2}, $$
    where $Z_i:= \sum_{k=1}^n \frac{1}{(\dermapi(\alpha)-\alpha_k)^2}$. 
\end{lemma}

\begin{proof}
For every $t\in\bR$ define $\alpha(t) := \alpha+ t e_j$.  Let $p_t(x):= \pmap{\alpha(t)}(x)$ and, to simplify notation, let $\delta_i(t) := \dermapi(\alpha(t))$. By definition $p_t'(\delta_i(t)) =0$, and therefore
\begin{equation}
\label{eq:cauchy_transform_zero}
0 = \frac{p_t'(\delta_i(t))}{p_t(\delta_i(t))} = \sum_{k=1}^n \frac{1}{\delta_i(t)-\alpha_k(t)}.
\end{equation}
On the other hand, differentiating the identity $\delta(t) = \dermap(\alpha(t))$ with respect to $t$, while applying the chain rule on the right-hand side together with the fact that $\alpha'(0)=e_j$, yields $\delta'(0) = \jacder(e_j)$ and therefore $$\delta_i'(0) = \adj{e}_i\jacder(e_j) = \pdelta_{ij}.$$ 
So, differentiating both sides of (\ref{eq:cauchy_transform_zero}) with respect to $t$ and evaluating at $t=0$ yields
$$\frac{1}{(\delta_i-\alpha_j)^2} - \sum_{k=1}^n \frac{\pdelta_{ij}}{(\delta_i-\alpha_k)^2}  =0$$
where we are using $\delta_i$ as a short-hand notation for $\delta_i(0)$. The proof then follows from rearranging the above equality and using that $\delta_i = \dermapi(\alpha)$. 
\end{proof}

From \Cref{lem:gauss_lucas_entries} we see that $\pdelta$ not only maps score vectors $\score(\alpha)$ to $\scoreder(\dermap(\alpha))$ (as  implied by \Cref{lem:score_to_score_der}), but in view of (\ref{eq:Gauss--Lucas_theorem}) one also has that
$$\dermap(\alpha)= \pdelta\cdot\alpha.$$

\subsection{Differentiators} 

The matrix defined below, obtained by taking entry-wise square roots of the Gauss--Lucas matrix, will be of use in the proof of \Cref{thm:fisher_monotonicity}. 

\begin{definition}
    Assume $\alpha\in \bR^n$ has distinct roots. Then, define $P_\alpha$ to be the  $(n-1)\times n$ matrix with entries
    $$(P_\alpha)_{ij} = \frac{1}{\sqrt{Z_i}}\cdot \frac{1}{\dermapi(\alpha)-\alpha_j}$$
    where $Z_i:= \sum_{k=1}^n \frac{1}{(\Omega_{\scriptscriptstyle\partial_x}^i(\alpha)-\alpha_k)^2}$.
\end{definition}

Another concept that will be useful in our proof is that of differentiators \cite{davis1959eigenvalues}. 

\begin{definition}[Differentiator]
    Let $P\in M_{(n-1)\times n}(\bR)$  with orthonormal rows, $A\in M_{n\times n}(\bR)$, and define $c_A(x):= \det(x-A)$ and $c_{PA\adj{P}}(x):= \det(x-PA\adj{P})$. We say that $P$ is a {\em differentiator} of $A$ if $c_{PA\adj{P}}(x)=\frac{1}{n}c_A'(x)$. 
\end{definition}

We now observe that $P_\alpha$ is a differentiator for diagonal matrices.  

\begin{lemma}
\label{lem:differentiator_matrix}
    For every $\alpha\in \bR^n$ with distinct entries, the matrix $P_\alpha$ has orthonormal rows and is the differentiator of any diagonal matrix. 
\end{lemma}

\begin{proof}
    That $P_\alpha$ has orthonormal rows all of which are orthogonal to $\bone_n$ is  stated without proof inside the proof of Lemma 3.3 in \cite{khavinson2011borcea}, so it might be folklore. Let us provide a proof here. First, to simplify notation let $\delta_i:= \dermapi(\alpha)$ and  $p(x):=\pmap{\alpha}(x)$.
    
    Now take $i, j$ any two distinct indices. To show that the $i$-th row of $P_\alpha$ is orthogonal to the $j$-th one, expand their inner product as follows
    \begin{align*}
    \frac{1}{\sqrt{Z_iZ_j}} \sum_{k=1}^n \frac{1}{\delta_i -\alpha_k} \cdot \frac{1}{\delta_j - \alpha_k} & = \frac{1}{(\delta_j - \delta_i)\sqrt{Z_iZ_j}} \sum_{k=1}^n \left( \frac{1}{\delta_i-\alpha_k}-\frac{1}{\delta_j - \alpha_k} \right) 
    \\ & = \frac{1}{(\delta_j - \delta_i)\sqrt{Z_iZ_j}} \left(\frac{p'(\delta_i)}{p(\delta_i)}- \frac{p'(\delta_j)}{p(\delta_j)}\right)
    \\ & =0,
    \end{align*}
   where in the last equality we used that $\frac{p'(\delta_i)}{p(\delta_i)} = \frac{p'(\delta_j)}{p(\delta_j)}=0$.

   Now, to show that $\bone_n$ is orthogonal to all the rows of $P_\alpha$ again use that, for every $i$, $p'(\delta_i)=0$ to get that
   $$0 = \frac{p'(\delta_i)}{\sqrt{Z_i}p(\delta_i)} = \sum_{j=1}^n \frac{1}{\sqrt{Z_i}(\delta_i-\alpha_j)} = \adj{e}_i P_\alpha \bone_n .$$

Now that we have established the above, we can conclude that $P_\alpha$ is a differentiator of any diagonal matrix via a standard argument (e.g.\ see \cite[Theorem 2.5]{pereira2003differentiators}). Indeed, let $D\in M_n(\bR)$ be any diagonal matrix and set $u:= \frac{1}{\sqrt{n}} \bone_n$. Let $c_D(x):=\det(x-D)$ and $c_{P_\alpha D \adj{P}_\alpha}(x) := \det(x-P_\alpha D\adj{P}_\alpha)$. Then, since $P_\alpha \bone_n =0$ and $P_\alpha$ has orthonormal rows, by the Schur complement formula we have that $\frac{c_{P_\alpha D\adj{P}_\alpha}(x)}{c_D(x)} = \adj{u}(x-D)^{-1}u $. On the other hand, since $D$ is diagonal, $\adj{u}(x-D)^{-1}u = \frac{1}{n}\trace((x-D)^{-1}) = \frac{1}{n}\frac{c_D'(x)}{c_D(x)}$.  Therefore
$$c_{P_\alpha D\adj{P}_\alpha}(x) = \frac{1}{n}c_D'(x),$$
as we wanted to show. 
\end{proof}

\begin{remark}[Finite free position]
\label{rem:finite_free_position}
   In analogy to the notion  of  finite free position introduced by Marcus  \cite[Definition 5.1]{marcus2021polynomial}, one can interpret \Cref{lem:differentiator_matrix} as saying that $P_\alpha$ is in free position from diagonal matrices. 
\end{remark}

\subsection{Bounding the second singular value of $\jacder$} 

We can now bound  the second largest singular value of $\jacder$.

\begin{proposition}\label{prop:jacder 2sv}
  For any $u\in \bR^n$ perpendicular to $\bone_n$ one has
  $$\|\jacder(u)\| \le \sqrt{\tfrac{n-2}{n}} \|u\|.$$
   %Therefore $\gs_2(\jacder)=\sqrt{\tfrac{n-2}{n}}.$
\end{proposition}
\begin{proof}
    From \Cref{lem:jacder doub stoch}, the top singular value of $\jacder$ is $\sqrt{\frac{n-1}{n}}$ and has associated  right singular vector $\bone_n$ and left singular vector $\bone_{n-1}$. Thus, it is  enough to show the first claim. 
   
    Take $u\in \bR^n$ with $\adj{u} \bone_n =0$ and set $g_u:=\pmap{u}$. Then, define the diagonal matrix $D_u:=\diag(u_1, \dots, u_n)$ and let $X:=P_\ga D_uP_\ga^\top$. Since $g_u(x)=\det(x-D_u)$, \Cref{lem:differentiator_matrix} gives that the roots of $g_u'(x)$ are the eigenvalues of $X$. So, if $w\in \bR^{n-1}$ is a vector of roots for $g_u'(x)$,   \Cref{obs:variancescaling}, on the one hand,  implies $\adj{w}\bone_{n-1} = \adj{u} \bone_n =0$,  and on the other hand
    $$\tfrac{n-2}{n}\norm{u}^2=\norm{w}^2 =\trace(X^2),$$ 
where the last equality follows from the fact that the $w_i$ are the eigenvalues of $X$. Now, we  compute $\trace(X^2)$ directly: 
    \[X(i,i)=\sum\limits_{j=1}^nP_\ga(i,j)D_u(j,j)P_\ga^\top(j,i)=\sum\limits_{j=1}^n\pdelta_{i, j} u_j=\adj{e}_i \jacder(u).\]
    It follows then that 
    \[\norm{\jacder (u)}^2=\sum\limits_{i=1}^nX(i,i)^2\le\norm{X}_{\mathrm{F}}^2=\tr(X^2)=\tfrac{n-2}{n}\norm{u}^2.\qedhere\]
\end{proof}

We are now equipped to prove \Cref{thm:fisher_monotonicity}. 

\begin{proof}
    [Proof of \Cref{thm:fisher_monotonicity}]
    First, \Cref{obs:orthogonality_score} tells us that  $\score(\ga)$ is perpendicular to $\bone_n$. So \Cref{lem:score_to_score_der} and \Cref{prop:jacder 2sv} combined yield
    \[\norm{\scoreder(\dermap(\ga))}^2=\norm{\jacder(\score(\ga))}^2\le\tfrac{n-2}{n}\norm{\score(\ga)}^2.\]
    In terms of finite free Fisher information, the above is equivalent to
    \[(n-1)(n-2)^2\Phi_{n-1}(p')\le(n-2)(n-1)^2\Phi_n(p)\]
    Let us briefly work out what $\widetilde{p'}$ has to be. If $\widetilde{p'}=c_*p'$ for some $c$, and $\var(p)=\var\lpr{\widetilde{p'}}=c^2\var(p')=\frac{n-2}{n-1}c^2\var(p)$, then $c=\sqrt{\frac{n-1}{n-2}}$. It then follows that $\Phi_{n-1}\lpr{\widetilde{p'}}=\frac{n-2}{n-1}\Phi_{n-1}(p')$, and the result is proved. 
\end{proof}

We remark that a direct calculation shows that the inequality is saturated at any Hermite polynomial $a_*H_n(x-b)$ for $a,b\in\R$.

\section{Stam's inequality via hyperbolicity}
\label{sec:stam_geometry_proof}

This section is devoted to proving the finite free Stam inequality stated in \Cref{thm:stams_inequality}. When dealing with $\boxplus_n$, as in the case of $\partial_x$, it is still possible to obtain formulas for $\jac$. However, these formulas are not as explicit as the one obtained in \Cref{lem:gauss_lucas_entries} and we did not see how to successfully apply them to get a proof.\footnote{That said, it is reasonable to expect that there exists a matrix-analytic proof in the same spirit as the one provided in \Cref{sec:matrix_analysis_fisher_monotonicity}, which would potentially uncover interesting phenomena related to matrices in finite free position.} So, here we adopt an entirely different approach using a result about hyperbolic polynomials (see \Cref{thm:baushke}) as an input. Then, in \Cref{sec:geometry_of_polynomials_fisher_monotonicity} we show that this  approach also provides an alternative proof of \Cref{thm:fisher_monotonicity}. We highlight that neither of these proofs requires the double stochasticity results proven in Lemmas \ref{lem:double_stochasticity_convolution} and \ref{lem:jacder doub stoch}.  

\subsection{Proof of Stam's inequality}
\label{sec:proof_of_stam}

Throughout this section we will fix simple ordered vectors $\alpha, \beta\in \bR^n$ and  use $$\jac: \bR^n \oplus \bR^n\to \bR^n$$ to denote the Jacobian of $\roots$ at $(\alpha, \beta)$. For every $i$ we will use $$\hessconv: (\bR^n \oplus \bR^n)\times (\bR^n \oplus \bR^n) \to \bR$$ 
to denote the Hessian of the map $\rootsi$ at $(\alpha, \beta)$.  It will also prove useful to define the subspace
$$\calV:=\{u\oplus v\in  \bR^n \oplus \bR^n : \adj{u} \bone_n = \adj{v} \bone_n =0\}.  $$

Our proof consists of three pieces:
\begin{enumerate}[(i)]
    \item (Relating score vectors)  First, we invoke \Cref{lem:score_to_score_convolution} which states that
    $$(a+b) \score(\roots(\alpha, \beta))=\jac(a\score(\alpha)\oplus b \score(\beta))$$
    for all $a, b\in \bR$. 
    \item (Relating the Jacobian and the Hessian) Then, below, in \Cref{lem:jacobian_vs_hessian_conv} we show that
    \begin{equation}
\label{eq:jacobian_vs_hessian}
   \|\jac(w)\|^2 = \|w\|^2- \sum_{i=1}^n \rootsi(\alpha, \beta) \hessconv (w, w)
     \end{equation}
    for all $w\in \calV$.   
    \item (Positivity) Finally, in \Cref{lem:psdness_conv}, by invoking the result of Bauschke et al.\ \cite{bauschke2001hyperbolic} discussed in \Cref{sec:geometry_of_polynomials}, we show that 
    \begin{equation}
    \label{eq:positivity}
    \sum_{i=1}^n \rootsi(\alpha, \beta) \hessconv (w, w)\ge 0.
        \end{equation}
    
\end{enumerate}
The proof is then easily concluded by putting all of the above together (see the last part of \Cref{sec:proof_of_stam}). 

We first  prove (\ref{eq:jacobian_vs_hessian}) by differentiating the identities in \Cref{prop:properties_of_ff_convolutions}. 

\begin{lemma}[Relating $\jac$ and $\hessconv$]
\label{lem:jacobian_vs_hessian_conv}
    For all $w\in \calV$ it holds that
    $$\|\jac(w)\|^2 = \|w\|^2- \sum_{i=1}^n \rootsi(\alpha, \beta) \hessconv (w, w).$$
\end{lemma}

\begin{proof}
Fix $w = u \oplus v\in \calV$ and define 
$$\alpha(t):= \alpha+ tu, \qquad \beta(t):=\beta+tv,  \qquad \text{and}\qquad  \gamma(t):=\roots(\alpha(t), \beta(t)),$$ and note that the variance additivity  from \Cref{prop:properties_of_ff_convolutions} (\ref{item:additivity}) implies that
$$m_2(\gamma(t))-m_1(\gamma(t))^2 = m_2(\alpha(t))+ m_2(\beta(t))-(m_1(\alpha(t))^2+m_1(\beta(t))^2),$$
for all $t\in \bR$. Now, the fact that $u \oplus v\in \calV$  implies that the means $m_1(\alpha(t))$ and $m_1(\beta(t))$ are a constant function of $t$ and therefore,  by the mean additivity in \Cref{prop:properties_of_ff_convolutions} (\ref{item:additivity}), the mean $m_1(\gamma(t))$ is also a constant function of $t$. So, differentiating the above equation twice with respect to $t$ we get
\begin{equation}
\label{eq:secderivatives}
 \left.\partial_t^2 m_2(\gamma(t))\right|_{t=0} = \left.\partial_t^2 \big(m_2(\alpha(t))+m_2(\beta(t))\big)\right|_{t=0} .
 \end{equation}
We will now explicitly compute both sides of the above equation and derive a new identity from equating both results. To this end, in what follows, for a function $f$ and a vector $a$, we use $D_af(x)$ to denote $\left.\partial_t f(x+ta)\right|_{t=0}$, the directional derivative of $f$ in the direction $a$. 

We begin by computing ($n$ times) the left-hand side of (\ref{eq:secderivatives}):
\begin{align}
\nonumber \left. n\, \partial_t^2m_2(\gamma(t))\right|_{t=0} & = \sum_{i=1}^n D_w^2(\rootsi(\alpha, \beta)^2) 
\\ \nonumber & = 2\sum_{i=1}^n \left( (D_w\rootsi(\alpha, \beta))^2+ \rootsi(\alpha, \beta)D_w^2\rootsi(\alpha, \beta)\right)
\\ \label{eq:secondderofgamma} & = 2\left(\|\jac(w)\|^2+ \sum_{i=1}^n\rootsi(w)  \hessconv(w, w) \right). 
\end{align}
Now compute ($n$ times) the right-hand side of (\ref{eq:secderivatives}):
\begin{align}
\nonumber n\,\partial_t^2(m_2(\alpha(t))+m_2(\beta(t))) & = \partial_t^2(\adj{(\alpha+tu)}(\alpha+tu)+\adj{(\beta+tv)}(\beta+tv)) 
 \\ \nonumber & = 2(\adj{u}u +\adj{v}v) 
 \\ \label{eq:secondderofalphanadbeta} & = 2\|w\|^2.
 \end{align}
The proof is then concluded by plugging (\ref{eq:secondderofgamma}) and (\ref{eq:secondderofalphanadbeta}) back into (\ref{eq:secderivatives}) and rearranging the terms. 
\end{proof}

Now we invoke \Cref{cor:bauschke} to prove a generalization of (\ref{eq:positivity}).

\begin{lemma}[Positivity]
\label{lem:psdness_conv}
 For any real numbers $c_1\ge \cdots \ge c_n$ and any $w\in \bR^n\oplus \bR^n$ we have 
  $$\sum_{i=1}^{n} c_i \hessconv (w, w)\ge 0.$$ 
\end{lemma}

\begin{proof}
First, by \Cref{prop:properties_of_ff_convolutions}, for any $\zeta, \eta\in \bR^n$ we have  
$$\sum_{i=1}^{n} \rootsi(\zeta, \eta) =  \sum_{i=1}^n (\zeta_i+\eta_i)$$
 and therefore the above is a linear function of $\zeta$ and $\eta$. So, viewing the sums in the above equation as functions of $\zeta$ and $\eta$, and taking the Hessian on each side at $(\zeta, \eta)=(\alpha, \beta)$, we get 
$$\sum_{i=1}^{n} \hessconv =0.$$
Hence $\sum_{i=1}^{n} c_i \hessconv = \sum_{i=1}^{n} (c_i+c) \hessconv$ for all $c\in \bR$ and therefore we can assume, without loss of generality, that $c_n\ge 0$. This will allow us to apply \Cref{cor:bauschke}.  

Now, define the  polynomial $f$ in $2n+1$ variables as
$$f_{\boxplus_n}(x, \zeta, \eta) := \frac{1}{n!} \sum_{\pi\in S_n} \prod_{i=1}^n(x-\zeta_i - \eta_{\pi(i)}),$$
and note that $\fconv$ is hyperbolic in the direction $e_1=(1, 0, \dots, 0)$. Indeed, let $x_0\in \bR$, $\zeta, \eta\in \bR^n$, and set $p(x):=\pmap{\zeta}(x)$ and $q(x):= \pmap{\eta}(x)$. Then 
\begin{equation}
\label{eq:g_and_conv}
g_{(x_0, \zeta, \eta)}(t):=\fconv(x_0+t, \zeta, \eta)= (p\boxplus_n q)(x_0+t).
\end{equation}
Then, since $(p\boxplus_nq)(x_0+t)$ is real-rooted so is $g_{(x_0, a, b)}(t)$, and we can conclude that $\fconv$ is hyperbolic in the direction $e_1$. 

Finally, if for every $\zeta, \eta\in \bR^n$ we use $\lambda_1(\zeta, \eta)\ge \cdots \ge \lambda_{n}(\zeta, \eta)$ to denote the roots of $g_{(0,\zeta, \eta)}(t)$, \Cref{cor:bauschke} implies that the function $\sum_{i=1}^n c_i \lambda_i(\zeta, \eta)$ is convex in $(\zeta, \eta)$. On the other hand, by (\ref{eq:g_and_conv}) we have that $\lambda_i(\zeta, \eta) = \rootsi(\zeta, \eta)$. So,  $\sum_{i=1}^n c_i \rootsi(\zeta, \eta)$ is convex as a function of $(\zeta, \eta)$. Differentiating this function twice at $(\zeta, \eta)=(\alpha, \beta)$ we get that $\sum_{i=1}^{n} c_i \hessconv$ is a positive semidefinite matrix, and the result follows.  
\end{proof}

We can now prove the following. 

\begin{proposition}
\label{prop:second_sing_val_bound_conv}
    For any $w\in \calV$ it holds that
    \begin{equation}
    \label{eq:norm_bound_on_V}
    \|\jac(w)\| \le \|w\|.
        \end{equation}
    In particular, for any $a, b \in \bR$ it holds that
    \begin{equation}
    \label{eq:main_score_ineq_conv}
    (a+b)^2\|\score(\roots(\alpha, \beta))\|^2 \le a^2 \|\score(\alpha)\|^2+b^2\|\score(\beta)\|^2.
        \end{equation}
\end{proposition}

\begin{proof}
Let $w\in \calV$.  Then
\begin{align*}
\|\jac(w)\|^2 & = \|w\|^2- \sum_{i=1}^n \rootsi(\alpha, \beta) \hessconv (w, w) && \text{\Cref{lem:jacobian_vs_hessian_conv} }
\\ & \le \|w\|^2 && \text{\Cref{lem:psdness_conv}}.
\end{align*}
Finally, by \Cref{obs:orthogonality_score} we have that $a\score(\alpha)\oplus b\score(\beta)\in \calV$. So (\ref{eq:main_score_ineq_conv}) follows from (\ref{eq:norm_bound_on_V}) together with \Cref{lem:score_to_score_convolution}. 
\end{proof}

The proof of \Cref{thm:stams_inequality} follows readily. 

\begin{proof}[Proof of \Cref{thm:stams_inequality}]
    Let $\alpha$ be a vector of roots for $p(x)$ and $\beta$ a vector of roots for $q(x)$. Let us first assume that  $\alpha$ and $\beta$ are simple and that $\lambda\neq 0$ and $\lambda\neq 1$.\footnote{Note when $\lambda\in \{0, 1\}$ the advertised inequality  trivially becomes an equality.} Then, (\ref{eq:main_score_ineq_conv}) applied to $\sqrt{\lambda}\alpha$ and $\sqrt{1-\lambda} \beta$ (which are again simple vectors) yields 
    $$(a+b)^2 \|\score(\roots(\sqrt{\lambda} \alpha, \sqrt{1-\lambda} \beta))\|^2 \leq \frac{a^2}{\lambda}\|\score(\alpha)\|^2 + \frac{b^2}{1-\lambda} \|\score(\beta)\|^2.$$
    So,  \Cref{thm:stams_inequality} follows  from taking   $a= \lambda$ and $b = 1-\lambda$ in the above inequality. 
    
    Now, if $\alpha$ or $\beta$ have repeated entries, we can take sequences of simple vectors $(\alpha^m)_{m=1}^\infty$ and $(\beta^m)_{m=1}^\infty$ in $\bR^n$ such that $\alpha = \lim_m \alpha^m$ and $\beta=\lim_m \beta^m$. Then, put $p_m := \pmap{\alpha^m}$ and $q_m := \pmap{\beta^m}$, and apply the derived inequality to get 
 $$    \Phi_n(\sqrt{\lambda}\,_* p_m \boxplus_n \sqrt{1-\lambda }\,_* q_m)\le \lambda \Phi_n(p_m) + (1-\lambda)\Phi_n(q_m).$$
 Then  take $m\to \infty$ to derive the desired result. 
\end{proof}

\subsection{Alternative proof of Fisher information monotonicity}
\label{sec:geometry_of_polynomials_fisher_monotonicity}

We now show that the same approach can be used to prove \Cref{thm:fisher_monotonicity}. Since all the steps are completely analogous to the proof given in \Cref{sec:proof_of_stam}, to minimize redundancy, here we give only a terse explanation. 

First, fix a simple ordered vector $\alpha\in \bR^n$ and let $\jacder$ be the Jacobian of $\dermap$ at $\alpha$, which we will think of as a linear transformation 
$$\jacder: \bR^n \to \bR^{n-1}.$$
Similarly, for every $i$, we will use $\hessder$ to denote the Hessian of $\dermapi$ at $\alpha$ and view it as a bilinear map
$$\hessder: \R^n \times \bR^n \to \bR.$$
We start by proving the analogue of \Cref{lem:jacobian_vs_hessian_conv}. 

\begin{lemma}[Relating the Jacobian and the Hessian]
\label{lem:jacobian_vs_Hessian}
   For any $u\in \bR^{n}$  orthogonal to  $\bone_n$, we have
    $$\|\jacder (u)\|^2 =\frac{n-2}{n}  \|u\|^2 - \sum_{i=1}^{n-1} \Omega_{\scriptscriptstyle\partial_x}^{i}(\alpha) \hessder(u, u)$$ 
\end{lemma}

\begin{proof}
    Define the trajectories $\alpha(t):= \alpha+t u$ and $\delta(t):= \dermap(\alpha(t))$. By \Cref{obs:variancescaling} we know that for every $t$ the following equation holds:
    \begin{equation}
    \label{eq:variance_for_all_t}
    m_2(\delta(t))-m_1(\delta(t))^2 = \tfrac{n-2}{n-1} \left( m_2(\alpha(t))-m_1(\alpha(t))^2 \right). 
        \end{equation}
        Then, as in the proof of \Cref{lem:jacobian_vs_hessian_conv}, we differentiate both sides of the above identity twice with respect to $t$, to obtain the desired result. We omit the explicit calculations since they are the same as in the proof of \Cref{lem:jacobian_vs_hessian_conv}.
\end{proof}

Now we prove the analogue of \Cref{lem:psdness_conv}. 

\begin{lemma}[Positivity]
\label{lem:psdness_derivative}
For any real numbers $c_1\ge \cdots \ge c_n$ and any $u\in \bR^n$ 
  $$\sum_{i=1}^{n-1} c_i \hessder (u, u)\ge 0,$$
  and moreover $\sum_{i=1}^{n-1} \hessder(u, u) =0$. 
\end{lemma}

\begin{proof}
By \Cref{obs:variancescaling}, for any $\zeta\in \bR^n$ 
$$\sum_{i=1}^{n-1} \dermapi(\zeta) = \frac{n-1}{n} \sum_{i=1}^n \zeta_i,$$
 and therefore the above is a linear function of $\zeta$. Differentiating twice yields 
$$\sum_{i=1}^{n-1} \hessder =0,$$
proving the second claim. On the other hand, it follows that $\sum_{i=1}^{n-1} c_i \hessder = \sum_{i=1}^{n-1} (c_i+c) \hessder$ for all $c\in \bR$ and therefore we can assume, without loss of generality, that $c_n\ge 0$,  which will allow us to   apply \Cref{cor:bauschke}.  

Now, define the multivariate polynomial
$$f_{\partial_x}(x, \zeta) := \sum_{i=1}^n \prod_{j : j\neq i}(x-\zeta_j),$$
and note that $\fder$ is hyperbolic in the direction $e_1=(1, 0, \dots, 0)$. Indeed, let $x_0\in \bR$, $\zeta\in \bR^n$, and $p(x):=\pmap{\zeta}(x)$, and note that 
\begin{equation}
\label{eq:g_and_der}
g_{(x_0, \zeta)}(t):=\fder(x_0+t, \zeta)= p'(x_0+t).
\end{equation}
Then, since $p(t)$ is real-rooted so is $g_{(x, a)}(t)$. Finally, apply \Cref{cor:bauschke} in the same way it was applied in the proof of \Cref{lem:jacobian_vs_hessian_conv}. 
\end{proof}

We now put the above together to provide an alternative proof of \Cref{prop:jacder 2sv}, which as discussed at the end of \Cref{sec:matrix_analysis_fisher_monotonicity}, suffices to readily prove \Cref{thm:fisher_monotonicity}. 

\begin{proof}[Alternative proof of \Cref{prop:jacder 2sv}]
    Take $u\in \bR^n$  perpendicular to $\bone_n$ and note that
    \begin{align*}
       \|\jacder (u)\|^2 &= \frac{n-2}{n}  \|u\|^2 - \sum_{i=1}^{n-1} \Omega_{\scriptscriptstyle\partial_x}^{i}(\alpha) \hessder(u, u) && \text{\Cref{lem:jacobian_vs_Hessian}}
       \\ & \le \frac{n-2}{n} \|u\|^2 && \text{\Cref{lem:psdness_derivative}.} \qedhere
    \end{align*}
\end{proof}

\subsection{Equality cases}
\label{sec:equality_cases_stam}

Now that Theorems \ref{thm:fisher_monotonicity} and \ref{thm:stams_inequality} have been established, the purpose of this section is to determine the equality cases for the corresponding inequalities.

\begin{theorem}[Equality cases]
\label{thm:equality_cases_Fisher}
  Let $p(x)$ and $q(x)$ be  degree $n$ monic real-rooted polynomials with simple roots.
   \begin{enumerate}[i)]
       \item \label{item:monotonicity} (Monotonicity) Let $\tilde{{p'}}$ be as in Theorem \ref{thm:fisher_monotonicity}. If $\Phi_{n-1}(\tilde{p'})=\Phi_n(p)$ then $p(x)$ is an affine transformation of $H_n(x)$. 
       \item \label{item:stam_ineq} (Stam's inequality) Let $\lambda\in (0, 1)$. If $ \Phi_n\big(\sqrt{\lambda}\,_* p \boxplus_n \sqrt{1-\lambda }\,_* q\big)= \lambda \Phi_n(p) + (1-\lambda)\Phi_n(q)$ then both $p(x)$ and $q(x)$ are affine transformations of $H_n(x)$. 
   \end{enumerate}
\end{theorem}

The key technical input in the proofs of parts \ref{item:monotonicity}) and \ref{item:stam_ineq}) in the above theorem are the facts that the second largest singular values of the Jacobians of the maps $\dermap$ and $\roots$, respectively, are simple. Concretely, as above, we will fix $\alpha, \beta\in \bR^n$ simple vectors and use $\jacder,  \jac$ to denote the Jacobnians of $\dermap$ and $\roots$, respectively, at $\alpha$ and $(\alpha, \beta)$, and we will use $\hessder$ and $\hessconv$ to denote the Hessians of $\dermapi$ and $\rootsi$, at the same corresponding points. We  prove the following. 

\begin{proposition}[Simple second singular values]
\label{prop:simple_sigma_derivative}
Let $\sigma_i(\jacder)$ and $\sigma_i(\jac)$ denote the $i$-th largest singular values of $\jacder$ and $\jac$, respectively. Then:
\begin{enumerate}[i)]
    \item \label{item:simple_sing_der}  $\sigma_2(\jacder) = \sqrt{\frac{n-2}{n}}$ is a simple singular value and   $\calpha:=(I_n-\frac{1}{n} \bone_n \bone_n^\top)\alpha$ is the associated right singular vector. 
    \item \label{item:simple_sing_conv} $\sigma_2(\jac)=1$ is a simple singular value and, if $\ceta:= (I_n-\frac{1}{n} \bone_n \bone_n^\top)\beta$, then $\calpha\oplus \ceta$ is the associated right singular vector. 
\end{enumerate}
\end{proposition}

\begin{remark}[Right singular vectors] Although it will not be of use in determining the equality cases, it may be worth noting that $\cdelta :=(I_{n-1}-\frac{1}{n-1} \bone_{n-1}\bone_{n-1}^\top) \dermap(\alpha)$ is the left singular vector of $\jacder$ associated to $\sigma_2(\jacder)$, and $\camma :=(I_n - \frac{1}{n} \bone_n \bone_n^\top ) \roots(\alpha, \beta)$ is the left singular vector of $\jac$ associated to $\sigma_2(\jac)$. 
\end{remark}

To derive Theorem \ref{thm:equality_cases_Fisher} from  Proposition \ref{prop:simple_sigma_derivative} we will combine the Hermite differential equation from Lemma \ref{lem:hermite_ode} with the following observation, which is  proven by a direct computation. 

\begin{observation}
\label{obs:score_as_double_derivative}
    If $p(x) = \pmap{\alpha}(x)$ then
    $$\frac{p''(\alpha_i)}{p'(\alpha_i)} = \sum_{j : j\neq i} \frac{1}{\alpha_i-\alpha_j}, \quad \forall i=1, \dots, n. $$
\end{observation}

For now, we will provide the proof that Proposition \ref{prop:simple_sigma_derivative} implies Theorem \ref{thm:equality_cases_Fisher}. The proof of Proposition \ref{prop:simple_sigma_derivative} is non-trivial and is deferred to  \cref{sec:simple_der} and \cref{sec:simple_conv}.

\begin{proof}[Proof of Theorem \ref{thm:equality_cases_Fisher}  using Proposition \ref{prop:simple_sigma_derivative}] We begin by proving \ref{item:monotonicity}). 
  Let $p(x):= \pmap{\alpha}(x)$. Since the inequality 
 in statement in Theorem \ref{thm:fisher_monotonicity} is invariant under translations by $\bone_n$, we can assume without loss of generality that $\alpha^\top \bone_n=0$, in which case $\calpha = \alpha$.    Now  note that the only inequality used in the proof of Theorem \ref{thm:fisher_monotonicity} is the one provided by Proposition \ref{prop:jacder 2sv} where $u = \score(\alpha)$. So, if  $\Phi_{n-1}(\tilde{p'}) = \Phi_n(p)$, we must have that
  \begin{equation}
  \label{eq:score_is_sing_val}
  \|\jacder(\score(\alpha))\| = \sqrt{\tfrac{n-2}{n}} \|\score(\alpha)\|.
   \end{equation}
  On the other hand, since $\score(\alpha)$ is orthogonal to $\bone_n$ and $\bone_n$ is the top singular vector of $\jacder$ and $\sigma_2(\jacder)= \sqrt{\frac{n-2}{n}}$, (\ref{eq:score_is_sing_val}) in conjunction with Proposition \ref{prop:simple_sigma_derivative}  implies that $ \alpha = c \score(\alpha)$ for some $c\in \bR$ with $c\neq 0$. 
  
  Now, Observation \ref{obs:score_as_double_derivative} tells us that $\frac{p''(\alpha_i)}{p'(\alpha_i)} = \score(\alpha)_i$, for every $i=1, \dots, n$. So, $\alpha = c \score(\alpha)$ yields $\alpha_i = c \frac{p''(\alpha_i)}{p'(\alpha_i)}$, or equivalently that $cp''(\alpha_i)-\alpha_ip'(\alpha_i)=0$, which, since $p(\alpha_i)=0$ can be rewritten as 
  $$cp''(\alpha_i)-\alpha_i p'(\alpha_i)+p(\alpha_i) =0. $$
  Now, the above holds for every $i$, which means that $\alpha_1, \dots, \alpha_n$ are roots of the equation $$cp''(x)-xp'(x)+np(x)=0.$$ 
  But, the left-hand side of the above equation is a polynomial of degree at most $n-1$, so it can only have $n$ roots if it is identically 0. So, by Lemma \ref{lem:hermite_ode} we conclude that $p(x) = aH_n(bx)$ for some $a, b \in \bR$. 

 The proof of part \ref{item:stam_ineq}) is  analogous to part \ref{item:monotonicity}).  First,  note that the only inequality used in the proof of Theorem \ref{thm:stams_inequality} is that from Proposition \ref{prop:second_sing_val_bound_conv} applied to the vector 
  $$w= \lambda\score( \sqrt{\lambda}\alpha)\oplus (1-\lambda) \score(\sqrt{1-\lambda}\beta)) = \sqrt{\lambda} \score(\alpha)\oplus \sqrt{1-\lambda}\score(\beta). $$
  So, if equality in Theorem \ref{thm:stams_inequality} holds, then Proposition \ref{prop:simple_sigma_derivative} implies that  $\sqrt{\lambda}\score( \alpha)\oplus \sqrt{1-\lambda} \score(\beta)$ is parallel to $\calpha \oplus \ceta$, and in particular we get that $\score(\alpha)$ is parallel to $\calpha$ and $\score(\beta)$ is parallel to $\ceta$, at which point we can proceed as in the proof of part \ref{item:monotonicity}). 
\end{proof}

\subsubsection{Second simple singular value: The derivative case }
\label{sec:simple_der}

\begin{definition}
    Let $f\in \bR[x_1, \dots, x_m]$ be a hyperbolic polynomial in the direction $u\in \bR^m$, and let $K(f, u)$ be its hyperbolicity cone. For $x\in \partial K(f, u)$ we will use $\mult(x)$ to denote the multiplicity of 0 as a root of the polynomial $h(t)=f(x+tu)$. Then,  we will define the set
    $$\partial^1 K(f, u) := \{ x \in \partial K(f, u) : \mult(x)=1\}.$$
\end{definition}

Now,  define the following polynomial in $\zeta\in \bR^n$,
$$g_{\lilder}(\zeta) :=  \sum_{i=1}^n \prod_{j : j \neq i} \zeta_j,$$
and note that $g_{\lilder}(\zeta)$ is hyperbolic in the direction $\bone_n$. Moreover, its hyperbolicity cone is given by $K(g_{\lilder}, \bone_n) = \{ \zeta\in \bR^n : \dermap^{n-1}(\zeta) > 0  \},$
and therefore its boundary admits the following description 
\begin{equation}
\label{eq:cone_boundary}
\partial K (g_{\lilder}, \bone_n) = \{\zeta\in \bR^n : \dermap^{n-1}(\zeta)=0\},
\end{equation}
and to lighten notation set  $K_{\lilder}:= K(g_{\lilder},\bone_n)$. In the sequel we will need the following particular case of a theorem of Renegar %\cite{renegar2006hyperbolic},
which articulates that the cone $K_{\lilder}$ does not have flat faces.

\begin{theorem}[Renegar]
\label{thm:renegar}
    If $\alpha\in \partial^1 K_{\lilder}$ then $\alpha$ is an extreme direction of $K_{\lilder}$.  That is, $\alpha$ cannot be expressed as a convex combination of non-parallel vectors in $\partial K_{\lilder}$. 
\end{theorem}

\begin{proof}
   Let $f(\zeta) := \zeta_1 \cdots \zeta_n$ so that its directional derivative $D_{\bone_n} p$ is $g_{\lilder}$. Then observe that $K_{\lilder}$ does not contain any 1-dimensional subspace, which is Renegar's  definition of being regular. We can then  apply \cite[Corollary 15]{renegar2006hyperbolic} to obtain the advertised claim. 
\end{proof}

Renegar's theorem will be combined with the following statement.

\begin{lemma}[Linearity]
\label{lem:linearity}
    Assume $u\in \bR^n$ is such that $\hessder(u, u)=0$ for all $i=1, \dots, n-1$. Then, there exists a  permutation $\tau\in S_{n-1}$ and an $\eps>0$ such that
    $$\dermapi(\alpha+t u) = \dermapi(\alpha) + t \dermap^{\tau(i)}(u), \quad \forall t \in (-\eps, \eps).$$
\end{lemma}

\begin{proof}
    Consider the polynomial
    $$f_{\lilder}(x, s, t):= \frac{1}{n} \sum_{j=1}^n \prod_{i: i \neq j} (x-s\alpha_i - t u_i).$$
    Note that $f_{\lilder}$ is homogeneous of degree $n-1$ and hyperbloic in the direction $(1, 0, 0)$. So, by the Helton--Vinnikov theorem (see Theorem \ref{thm:Helton-Vinnikov} above) we know that there are real symmetric matrices $A$ and  $B$ of dimension $n-1$ such that
    $$f_{\lilder}(x, s, t) = \det(x-sA-tB). $$
    Moreover, note that without loss of generality we can assume that $A$ is diagonal with its diagonal entries appearing in descending order. Now let $\lambda_1(t)\geq \cdots \geq \lambda_{n-1}(t)$ be the eigenvalues of $A+tB$, so that 
    $$f_{\lilder}(x, 1, t) = \prod_{i=1}^{n-1}(x-\lambda_i(t)).$$
   On the other hand, from the definition of $f_{\lilder}(x, s, t)$, we have that 
    $$f_{\lilder}(x, 1, t) = \frac{1}{n} \partial_x \pmap{\alpha+t u} = \prod_{i=1}^{n-1}(x-\dermapi(\alpha+tu)).$$
    So,  we have that $\lambda_i(t) = \dermapi(\alpha + t u)$ for all $i$ and therefore the hypothesis  $\hessder(u, u) =0$ yields $\lambda_i''(0)=0$. On the other hand, by Hadamard's second variation formula applied to the matrix $A$ (which has eigenvectors $e_1, \dots, e_{n-1}$) we get that 
    $$\lambda_i''(0) = 2 \sum_{j : j\neq i} \frac{|e_j^* B e_i|^2}{\lambda_i(0)-\lambda_j(0)}. $$
    Now, since $\lambda_1(0)>\lambda_j(0)$ for all $j$, when $i=1$ all the terms in the right-hand side of the above equation are non-negative, and therefore $\lambda_1''(0)=0$ implies that $e_j^* B e_1 =0$ for all $j\neq 1$. Proceeding by induction over $i$ (while using that $B$ is symmetric), in the same way we get that $e_j^*B e_i=0$ for all $i\neq j$. Hence, $B$ is  diagonal. 
    
    Now let $\mu_1\geq \cdots \geq \mu_{n-1}$ be the eigenvalues of $B$. Then 
    $$\prod_{i=1}^{n-1} (x-\dermapi(u)) = f_{\lilder}(x, 0, 1) = \det(x-B) = \prod_{i=1}^{n-1}(x-\mu_i),$$
    and it follows that $\mu_i= \dermapi(u)$. So, since $A$ and $B$ are diagonal, and the $\dermapi(u)$ appear in some order on the diagonal $B$, we get that for some fixed permutation $\tau\in S_{n-1}$ one has that
$$\prod_{i=1}^{n-1}(x-\dermapi(\alpha + t u))=f_{\lilder}(x, 1, t) = \det(x-A-tB)=\prod_{i=1}^{n-1}(x-\dermapi(\alpha)- t \dermap^{\tau(i)}(u)).$$   
   Then, since $\dermap(\alpha)$ is a simple vector, for $t$ in a neighborhood of 0 it will hold that $\dermapi(\alpha+tu) = \dermapi(\alpha)+ t\dermap^{\tau(i)}(u)$. 
\end{proof}

With the above at hand Proposition \ref{prop:simple_sigma_derivative} \ref{item:simple_sing_der}) follows readily. 

\begin{proof}[Proof of Proposition \ref{prop:simple_sigma_derivative} \ref{item:simple_sing_der})]
First, Lemma \ref{lem:jacder doub stoch} implies that $\bone_n$ is the top singular vector of $\jacder$ and $\sigma_1(\jacder) = \sqrt{\frac{n-1}{n}}$. Second, Proposition \ref{prop:jacder 2sv} implies that $\sigma_2(\jacder)\leq \sqrt{\frac{n-2}{n}}$. Now, differentiating the scaling identity $\dermap(t\alpha) = t \dermap(\alpha)$ with respect to $t$ we get that $\jacder(\alpha) =  \dermap(\alpha) := \delta$. And, combining this with the fact that $\jacder \bone_n$ is parallel to $\bone_{n-1}$ we can conclude that
$$\jacder(\calpha) = \cdelta,$$
where $\cdelta := (I_{n-1}- \frac{1}{n-1}\bone_{n-1}\bone_{n-1}^\top) \delta$. On the other hand, since $\dermap(\calpha) = \cdelta$, Observation \ref{obs:variancescaling} implies that $\|\cdelta\|^2 = \frac{n-2}{n} \|\calpha\|^2$, and therefore $\|J(\calpha)\| = \sqrt{\frac{n-2}{n}} \|\calpha\|$. And, since $\calpha$ is orthogonal to $\bone_n$, it follows that $\calpha$ is the right singular vector of $\jacder$ associated to $\sigma_2(\jacder) = \sqrt{\frac{n-2}{n}}$.  

      Now we will show that $\sigma_2(\jacder)$ is simple. Proceeding by contradiction assume that there exists a $u\in \bR^n$  orthogonal to $\bone_n$ (the top singular vector of $\jacder$), non-parallel to $\calpha$, and such that $\|\jacder u\| = \frac{n-2}{n} \|u\|$. Now, the latter combined with  Lemma \ref{lem:jacobian_vs_Hessian} implies that
    \begin{equation}
    \label{eq:adding_to_zero}
    \sum_{i=1}^{n-1} \Omega_{\scriptscriptstyle\partial_x}^{i}(\alpha) \hessder(u, u) =0.
        \end{equation}
    On the other hand, we can rewrite the left-hand side in the above equation as the following telescoping sum
    \begin{equation}
    \label{eq:eq_to_zero}
    \sum_{i=1}^{n-1} \Omega_{\scriptscriptstyle\partial_x}^{i}(\alpha) \hessder(u, u) = \dermap^{n-1}(\alpha)\sum_{i=1}^{n-1} \hessder(u, u) + \sum_{j=1}^{n-1} (\dermap^j(\alpha) - \dermap^{j+1}(\alpha)) \left(\sum_{i=1}^j \hessder(u, u) \right).
     \end{equation}
    The point being that, by Lemma \ref{lem:psdness_derivative}, the first sum in the right-hand side of (\ref{eq:eq_to_zero}) is zero, and each of the terms in the second sum are non-negative and therefore zero by (\ref{eq:adding_to_zero}). So, since $\dermap(\alpha)$ is simple, inductig on $j$ gives us that $\hessd^j(u, u)=0$ for all $j=1, \dots, n-1$. At this point, we can apply Lemma \ref{lem:linearity} to obtain 
\begin{equation}
\label{eq:loc_linearity}
\dermap^{n-1}(\alpha + tu) = \dermap^{n-1}(\alpha) + t \dermap^j(u)
\end{equation}
for some fixed $j\in [n-1]$ and for all $t\in (-\eps, \eps)$ for some $\eps>0$. On the other hand, since 
\begin{equation}
\label{eq:shift}
\dermap(\zeta+s\bone_n) = \dermap(\zeta) + s \bone_n, 
\end{equation}
for all $s\in \bR$ and all $\zeta\in \bR^n$, we can find $s_1, s_2\in \bR$ such that $\dermap^{n-1}(\alpha + s_1 \bone_n) =0$ and $\dermap^{n-1}(u+ s_2 \bone_n)=0$. Then, set $\alpha' = \alpha + s_1 \bone_n$ and $u'= u + s_2 \bone_n$. The point here is that the  shifts have been chosen to ensure  that $\dermap^{n-1}(\alpha'+tu')=0$ for all $t$ in a neighborhood of 0. Indeed
\begin{align*}
    0 & = \dermap^{n-1}(\alpha')+ t \dermap^{j}(u')
    \\ &= \dermap^{n-1}(\alpha) + t \dermap^{j}(u) + s_1+ts_2 && \text{by (\ref{eq:shift})}
    \\ & = \dermap^{n-1}(\alpha+ t u) + s_1+ts_2 && \text{by (\ref{eq:loc_linearity})}
    \\ & = \dermap^{n-1}(\alpha'+t u')  && \text{by (\ref{eq:shift}).}
\end{align*}
Now, since $\alpha$ is simple, so is  $\dermap(\alpha'+tu')$ for all $t$ in a neighborhood of 0. This, combined with $\dermap^{n-1}(\alpha'+tu')=0$, yields that $\alpha'+tu' \in \partial^1 K_{\lilder}$ for all $t$ in a neighborhood of 0. Now, since $\alpha'$ and $u'$ are non-parallel (as $u$ is perpendicular to $\bone_n$ and non-parallel to $\calpha$), we get that $\alpha'$ is not an extreme direction of $\partial^1 K_{\lilder}$, contradicting the statement from Theorem \ref{thm:renegar}. 
\end{proof}

\subsubsection{Simple second singular value:  the finite free convolution case}
\label{sec:simple_conv}

The same approach as in  \cref{sec:simple_der} will allow us to establish Proposition \ref{prop:simple_sigma_derivative} \ref{item:simple_sing_der}). Importantly,  we will need the analog of Renegar's  theorem for finite free convolutions, which we could not find in the literature. Establishing this result  will be the emphasis of this section. 

Define the polynomial $g_{\lilconv}\in \bR[\xi_1, \dots, \xi_n, \zeta_1, \dots, \zeta_n]$ by
$$g_{\lilconv}(\xi, \zeta) :=  \sum_{\sigma\in S_n} \prod_{i=1}^n (\xi_i+\zeta_{\sigma(i)}).$$
Note that $g_{\lilconv}(\xi, \zeta)$ is hyperbolic in the direction $(\bone_n, 0_n)$, where $0_n$ denotes the all-zeroes vector in $\bR^n$, and  its hyperbolicity cone is given by $\cone:= K(g_{\lilconv}, (\bone_n, 0_n)) =\{ (\xi,\zeta)\in \bR^n\times \bR^n : \roots^{n}(\xi,\zeta) > 0  \},$ so its boundary is
\begin{equation}
\label{eq:cone_boundary_convo}
\bcone = \{(\xi, \zeta)\in \bR^n\times \bR^n : \roots^{n}(\xi, \zeta)=0\}.
\end{equation}
Now, given $\xi \in \bR^n$ we will use  $\xi_{\min} := \min_{i\in [n]} \xi_i$ to denote the value of the minimal entries of $\xi$ and use $\mult_{\min}(\xi)$ to denote the number of entries of $\xi$ that are equal to $\xi_{\min}$. The analog of $\partial^1K_{\lilder}$ in this setting will be 
$$\partial'K_{\lilconv}:= \{(\xi, \zeta)\in \bR^n \times \bR^n : \roots^n(\xi, \zeta)=0 \quad \text{and} \quad \mult_{\min}(\xi)+ \mult_{\min}(\eta) \leq n\}. $$

We will prove the following. 

\begin{theorem}[Strictly curved boundary]
\label{thm:curved_cone_conv}
    If $\rho\in \partial'\cone$ then $\rho$ is an extreme direction of $\cone$. That is, $\rho$ cannot be expressed as a convex combination of non-parallel vectors in $\partial \cone$.
\end{theorem}

To establish this theorem we will first need a preliminary lemma whose proof relies on a result of Fujie \cite{fujie2026regularity} that characterizes the cases when the output of the map $\roots$ is non-simple.

\begin{lemma}
\label{lem:sum_of_minimals}
Let $\alpha, \beta \in \bR^n$. If $\roots^n(\alpha, \beta) = \roots^{n-1}(\alpha, \beta)=0$ then 
$$\alpha_{\min}+\beta_{\min}=0. $$
\end{lemma}

\begin{proof}
 Let $\gamma:= \roots(\alpha, \beta)$ and set 
   $$\mu_\alpha := \frac{1}{n} \sum_{i=1}^n \delta_{\alpha_i}, \quad \mu_\beta := \frac{1}{n} \sum_{i=1}^n \delta_{\beta_i}, \quad \text{and} \quad \mu_\gamma := \frac{1}{n} \sum_{i=1}^n \delta_{\gamma_i}. $$
   Since $0$ is a non-simple value of $\roots(\alpha, \beta)$, by \cite[Theorem 1.1]{fujie2026regularity} we get that there exists an $s\in \bR$ such that
   $$\mu_\gamma\{0\} = \mu_{\alpha}\{s\}+ \mu_\beta\{-s\}-1$$
   and also 
   $$F_\gamma(0) = F_\alpha(s)+F_\beta(-s)-1,$$
   where $F_\alpha, F_\beta,$ and $F_\gamma$ are the CDF's of $\mu_\alpha, \mu_\beta$ and $\mu_\gamma$.  Now, since $\roots^n(\alpha, \beta) = 0$ we have that $\mu_\gamma\{0\}=F_\gamma(0)$, which, in view of the above equations yields 
   $$\mu_{\alpha}\{s\}+ \mu_\beta\{-s\}= F_\alpha(s)+F_\beta(-s).$$
   But, since by definition we must have $F_\alpha(s)\geq \mu_\alpha\{s\}$ and $F_\beta(-s) \geq \mu_\beta\{-s\}$, it follows that $\mu_\alpha\{s\} = F_\alpha(s)$ and $\mu_\beta\{-s\} = F_\beta(-s)$, or equivalently $\alpha_{\min} =s$ and $\beta_{\min} =-s$. 
\end{proof}

We can now prove the theorem. Our argument is inspired by the proof of Theorem 14 in \cite{renegar2006hyperbolic}. 

\begin{proof}[Proof of Theorem \ref{thm:curved_cone_conv}]
Proceed by contradiction and assume that there is an interval of positive length $I\subset \overline{\cone}$ that contains $\rho$ and such that the unique line containing $I$ does not go through the origin. Since $\cone$ is convex and $\rho \in \partial'\cone \subset \partial \cone$, we must in fact have that $I \subset \partial \cone$. Now, since $\partial \cone$ is closed under multiplication by a positive scalar and $I \subset \partial \cone$, it follows that $\rho$ is contained in a non-trivial 2-dimensional sector which itself is contained in $\partial \cone$. So, let $\calS$ be a maximal 2-dimensional sector satisfying that $\rho \in \calS$ and $\calS \subset \partial \cone$. Since $\partial \cone$ does not contain any 1-dimensional subspace, we must have that $\calS$ is delimited by two rays $\ell_1$ and $\ell_2$ stemming from the origin such that the angle between them is less than $\pi$. Let $\ell$ be a third line, coplanar with $\calS$, going through $\rho$ and intersecting $\ell_1$ and $\ell_2$ in $\upsilon^1$ and $\upsilon^2$, respectively. It follows that $[\upsilon^1, \upsilon^2]$ is a maximal interval contained in $\partial \cone$. 

Now, by (\ref{eq:cone_boundary_convo}), $[\upsilon^1, \upsilon^2]\subset \partial\cone$ implies that $\roots^n(t\upsilon^1+(1-t)\upsilon^2)=0$ for all $t\in [0, 1]$, and therefore 
$$h(t):= g_{\lilconv}(t\upsilon^1+(1-t)\upsilon^2)=0, \quad \forall t \in [0,1].$$
Now, since $h(t)$ is a polynomial,  in fact, one must have $h \equiv 0$, and therefore
$$g_{\lilconv}(t\upsilon^1+(1-t)\upsilon^2)=0, \quad \forall t\in \bR.$$
Hence, for every $t\in \bR$ there must be an $i\in [n]$ (depending on $t$) such that $\rootsi(t\upsilon^1+(1-t)\upsilon^2)=0$. Moreover, since for $t\in \bR \setminus [0, 1]$ one has $t\upsilon^1+(1-t)\upsilon^2 \notin \partial \cone$, we must have $i\neq n$. So, taking $t \uparrow 0$ and $t\downarrow 1$, a continuity argument then implies that 
$$\roots^n(\upsilon^1)= \roots^{n-1}(\upsilon^1)=0=\roots^n(\upsilon^2)= \roots^{n-1}(\upsilon^2).$$

Therefore, if $\upsilon^1=(\alpha^1, \beta^1)$ and $\upsilon^2=(\alpha^2, \beta^2)$,   Lemma \ref{lem:sum_of_minimals} yields that 
\begin{equation}
\label{eq:zero_bottom}
\alpha^1_{\min}+\beta^1_{\min}=0= \alpha^2_{\min}+\beta^2_{\min}.
\end{equation}
So, if $\rho=(\alpha, \beta)$, since $\rho = t\upsilon^1+(1-t)\upsilon^2$ for some $t\in (0, 1)$, (\ref{eq:zero_bottom}) implies that
\begin{equation}
\label{eq:lower_bound}
\alpha_{\min}+\beta_{\min}\geq 0.
\end{equation}
Now let $p(x) := \pmap{\alpha}$ and $q(x):= \pmap{\beta}(x)$, we will show that $p\boxplus_n q$ has no roots on $(-\infty, 0]$. Let $x\leq 0$ and note that by (\ref{eq:lower_bound})
$$\prod_{i=1}^n (\alpha_i+\beta_{\sigma(i)}-x)\geq 0, \quad \forall \sigma\in S_n.$$
 Moreover, since we are assuming $\rho\in \partial' \cone$ we have that $\mult_{\min}(\alpha)+\mult_{\min}(\beta)\leq n$ and therefore, there exists a $\tau\in S_n$ such that for all $i\in [n]$ either $\alpha_i > \alpha_{\min}$ or $\beta_{\tau(i)} > \beta_i$, and therefore $\prod_{i=1}^n (\alpha_i+\beta_{\tau(i)}-x)>0$. Hence 
 $$(-1)^n p(x)\boxplus_n q(x) = \frac{1}{n!} \sum_{\sigma\in S_n} \prod_{i=1}^n (\alpha_i+\beta_{\sigma(i)}-x)>0,$$
 for all $x\leq 0$. In particular $\roots^n(\alpha, \beta) >0$, contradicting the assumption that $(\alpha, \beta)=\rho \in \partial\cone$. 
\end{proof}

We can now proceed as in the derivative case.

\begin{lemma}[Linearity]
\label{lem:linearity_conv}
    If $w = u\oplus v\in \bR^n \oplus \bR^n$ is such that $\hessconv(w, w)=0$ for all $i=1, \dots, n$ then, there exists a  permutation $\tau\in S_{n}$ and an $\eps>0$ such that
    $$\rootsi((\alpha, \beta)+t (u, v)) = \rootsi(\alpha, \beta) + t \roots^{\tau(i)}(u, v), \quad \forall t \in (-\eps, \eps).$$
\end{lemma}

\begin{proof} 
   The proof proceeds exactly as that of Lemma \ref{lem:linearity} but working with the polynomial 
    $$f_{\lilconv}(x, s, t):= \frac{1}{n!} \sum_{\sigma\in S_n} \prod_{i=1}^n (x-s(\alpha_i+\beta_{\sigma(i)})-t(u_i+v_{\sigma(i)})),$$
    instead, which is hyperbolic in the direction $(1, 0, 0)$. 
\end{proof}

The proof of part \ref{item:simple_sing_conv}) of Proposition \ref{prop:simple_sigma_derivative} is very similar to that of part \ref{item:simple_sing_der}). 

\begin{proof}[Proof of Proposition \ref{prop:simple_sigma_derivative} \ref{item:simple_sing_conv})]
  Recall the definition of the set $\calV:=\{u\oplus v\in  \bR^n \oplus \bR^n : \adj{u} \bone_n = \adj{v} \bone_n =0\}$.  Now, by \Cref{obs:norm_bound_jacobian} we know that $\bone_n \oplus \bone_n$ is the top right singular vector of $\jac$. On the other hand, $(\bone_n \oplus \bone_n)^\perp = \calV \oplus \Span(\bone_n\oplus -\bone_n )$ and $\jac(\bone_n, -\bone_n) =0$. So, Proposition \ref{prop:second_sing_val_bound_conv} implies that $\sigma_2(\jac)\le 1$. Now, differentiating the identity $\roots((1+t)(\alpha, \beta)) = (1+t)\roots(\alpha, \beta)$ with respect to $t$, we get that $\jac(\alpha\oplus \beta) = \roots(\alpha, \beta)=: \gamma.$ Then, putting $\camma:=\lpr{I_n-\tfrac{1}{n}\bone_n\bone_n^\top}\gamma$ and using that $\jac(a\bone_n\oplus b\bone_n)= (a+b)\bone_n$, we get that 
  $$\jac(\calpha \oplus \ceta  ) = \camma.$$
  So, combining this  with $\sigma_2(\jac)\le 1$, $\jac(\calpha\oplus \ceta) = \camma$, and (by \Cref{prop:properties_of_ff_convolutions} \ref{item:additivity})) $\|\calpha\|^2+ \|\ceta\|^2=\|\camma\|^2$, we get that $\sigma_2(\jac) =1,$ and that $\calpha\oplus \ceta$ is an associated right singular vector.
  
    To prove that $\sigma_2(\jac)$ is simple we proceed by contradiction and assume that there is a $w = u\oplus v \in \calV$ such that $w$ is linearly independent from $\calpha\oplus \ceta$ and $\|\jac(w)\| = \|w\|$. By Lemma \ref{lem:jacobian_vs_hessian_conv} this implies that
    $$\sum_{i=1}^n \rootsi(\alpha, \beta) \hessconv(w, w)=0.$$
    Then, replicating the argument from the proof of Proposition \ref{prop:simple_sigma_derivative} \ref{item:simple_sing_der}) with the corresponding telescoping sum we get that $\hessconv(w, w)=0$ for all $i=1, \dots, n$. Now, exactly as in the proof of Proposition \ref{prop:simple_sigma_derivative} \ref{item:simple_sing_der}) we can use Lemma \ref{lem:linearity_conv} to produce a non-trivial segment contained in $\partial' \cone$, contradicting the fact proven in Theorem \ref{thm:curved_cone_conv} that all points in $\partial' \cone$ are extreme points of the cone. 
\end{proof}

\section{Entropy inequalities}
\label{sec:entropy_ineq}

In this section we prove  monotonicity  under differentiation for the finite free entropy, which will follow from the corresponding statement for the finite free Fisher information (\Cref{thm:fisher_monotonicity}). We also prove the finite free entropy power inequality, which arises from the finite free Stam inequality (\Cref{thm:stams_inequality}). Both results will be derived via  the  finite free de Bruijn identity, which is the tool that enables  translating  results about  Fisher into results about entropy. 

As above, $H_n(x)$ will denote the $n$-th monic (probabilist) Hermite polynomial, which satisfies that $\var(H_n) =n-1$. For the discussion below it will be convenient to deal with their normalized versions (of variance 1)
$$\check{H}_n := (n-1)^{-1/2}_* H_n,$$
which are the analogue to the standard Gaussian in classical probability (see \cite[\S 6.2.2]{marcus2021polynomial}).

The finite free de Bruijn identity connects the finite free entropy and the Fisher information via differentiation and follows trivially from the description of the root dynamics under the heat flow given in \Cref{lem:score_and_derivatives}. 

\begin{lemma}[Finite free de Bruijn identity]
\label{thm:ffdb}
    Let $p(x)$ be a real-rooted polynomial of degree $n$ and  $s\ge0$. Then, 
    $$\frac{1}{2} \Phi_n(p\boxplus_n \sqrt{s}_* \check{H}_n) =  \partial_t\left. \rchi_n\lpr{p\boxplus_n \sqrt{t}_*\check{H}_n}\right|_{t=s}.$$
\end{lemma}

This result is implicitly present in a blog post of Tao \cite{tao2017blogpost}, explicitly present in an unpublished note by Marcus \cite{marcus2018note}, and was originally suggested to us by D. Shlyakhtenko. Below, for the reader's convenience, we provide the explicit computations involved in the proof.

\begin{proof}
First, by definition of $\check{H}_n$ and by \eqref{eq:heat flow of poly} we have that  
\begin{align*}
p(x)\boxplus_n \sqrt{t}_*\check{H}_n(x) & = p(x) \boxplus_n \left(\tfrac{t}{n-1}\right)^{n/2}H_n\left(\sqrt{\tfrac{n-1}{t}} \, x \right)
\\ & = \exp \left( - \tfrac{t}{2(n-1)} \partial_x^2 \right) p(x)
\\ & =: p_t(x).
\end{align*}
So now, if $\alpha_1(t), \dots, \alpha_n(t)$ are the roots of $p_t(x)$ then, a change of variables and \Cref{lem:score_and_derivatives} yield for every $t$ the identity
\begin{equation}
\label{eq:dervandhilb}
\alpha_i'(t) = \frac{1}{n-1}\sum_{j : j \neq i} \frac{1}{\alpha_i(t)-\alpha_j(t)}.
\end{equation}
On the other hand, by definition:
\begin{align*}
\partial_t \rchi_n(p_t)\big|_{t=s} & = \frac{1}{n(n-1)} \sum_{i \neq j} \frac{\alpha_i'(s)-\alpha_j'(s)}{\alpha_i(s)-\alpha_j(s)} 
\\ & = \frac{2}{n(n-1)} \sum_{i=1}^n \sum_{j: j\neq i} \frac{\alpha_i'(s)}{\alpha_i(s)-\alpha_j(s)} 
\\ & = \frac{2}{n} \sum_{i=1}^n \lpr{\ga_i'(s)}^2 && \text{by \eqref{eq:dervandhilb},}
\end{align*}
which combined with \eqref{eq:dervandhilb} yields the advertised result. 
\end{proof}

\subsection{Monotonicity of finite free entropy under differentiation}\label{sec:entropy_monotonicity}

The entropy monotonicity theorem (\Cref{thm:monotonicity})  will be derived via a combination of \Cref{thm:fisher_monotonicity} and \Cref{thm:ffdb}; the use of the latter will be to integrate over $\R_+$, as in the following result. 

\begin{lemma}\label{prop:ffdb int}
    Let $p(x)$ be a real-rooted polynomial of degree $n$. Then, 
    \[\rchi_n(p)=\rchi_n\lpr{\check{H}_n} + \frac{1}{2} \int_0^\infty \left(\frac{1}{1+t}- \Phi_n\big(p\boxplus_n \sqrt{t}_* \check{H}_n\big)\right) \d t . \]
\end{lemma}

\begin{proof}
    To simplify notation set $p_t:= p\boxplus_n \sqrt{t}_* \check{H}_n$. Now, by applying \Cref{thm:ffdb}, for every $T>0$ we have
    \begin{align}
  \nonumber  \rchi_n(p) & = \rchi_n(p_T) - \int_0^T \partial_t \rchi_n(p_t) \d t 
    \\ & \label{eq:diff} = \rchi_n(p_T)- \frac{1}{2} \int_0^T \Phi_n(p_t) \d t. 
       \end{align}
       We would now like to understand the asymptotic behavior of this difference as $T\to \infty$. 
       \begin{align*}
      \rchi_n(p_T) & = \rchi_n\lpr{\sqrt{T}_*\lpr{\lpr{T^{-1/2}}_* p \boxplus_n \check{H}_n }}
           \\ & = \frac{1}{2}\log(T) + \rchi_n\lpr{\lpr{T^{-1/2}}_* p \boxplus_n \check{H}_n} && \text{by \eqref{eq:rescaling}}
           \\ & = \frac{1}{2} \int_0^{T-1}\frac{1}{1+t} \d t + \rchi_n\lpr{\lpr{T^{-1/2}}_* p \boxplus_n \check{H}_n}.
       \end{align*}
       So, we have 
       \begin{align*}
           \rchi_n(p_T)- \frac{1}{2} \int_0^T \Phi_n(p_t) \d t = \frac{1}{2} \int_0^T\left(\frac{1}{1+t}-\Phi_n(p_t)\right) \d t + \rchi_n\lpr{\lpr{T^{-1/2}}_*p \boxplus_n \check{H}_n} + \frac{1}{2}\int_{T-1}^T \frac{1}{1+t} \d t. 
       \end{align*}
       Then, to conclude combine the above with \eqref{eq:diff} and take the limit as $T\to \infty$. 
\end{proof}

We can now prove the theorem.

\begin{proof}[Proof of \Cref{thm:monotonicity}]
    Once again we shall use the notation $p_t:= p\boxplus_n \sqrt{t}_* \check{H}_n= e^{-t\frac{\partial_x^2}{2(n-1)}} p$.     \Cref{thm:fisher_monotonicity} and \Cref{prop:ffdb int} imply that
    \begin{equation}
    \label{eq:upper_bound_with_qt}
        \rchi_n(p)\le\rchi_n\lpr{\check{H}_n}+\frac{1}{2}\int_0^\infty \left(\frac{1}{1+t}-\Phi_{n-1}(q_t)\right)\d t,
    \end{equation}
where $q_t$ is the rescaling of $(p_t)'$ so that $\var(q_t) = \var(p_t)$. Explicitly, Observation \ref{obs:variancescaling} tells us that 
$$q_t = \sqrt{\tfrac{n-1}{n-2}}_* (p_t)'.$$ 
Now, note that 
\begin{equation}
\label{eq:diff_p_t}
(p_t)' = \partial_x e^{-t\frac{\partial_x^2}{2(n-1)}} p = e^{-t\frac{\partial_x^2}{2(n-1)}} p' = e^{-\frac{(n-2)t}{(n-1)}\frac{\partial_x^2}{2(n-2)}} p' = p' \boxplus_{n-1} \sqrt{\tfrac{n-2}{n-1}t}_* \check{H}_{n-1}.
\end{equation}
And, therefore 
\begin{equation}
\label{eq:formula_for_qt}
q_t = \sqrt{\tfrac{n-1}{n-2}}_* \left(p' \boxplus_{n-1} \sqrt{\tfrac{n-2}{n-1}t}_* \check{H}_{n-1}\right) = \tilde{p'} \boxplus_{n-1} \sqrt{t}_* \check{H}_{n-1},
\end{equation}
where, again by Observation \ref{obs:variancescaling}, $\tilde{p'}:= \sqrt{\tfrac{n-1}{n-2}}_* p'$ is such that $\var(\tilde{p'}) = \var(p)$. 

To conclude the proof, substitute \eqref{eq:formula_for_qt} back into the right-hand side of \eqref{eq:upper_bound_with_qt} to get
$$\rchi_n(p) \le \rchi_{n}[\check{H}_n] + \frac{1}{2}\int_0^\infty \left(\frac{1}{1+t}- \Phi_{n-1}\big(\tilde{p'}\boxplus_{n-1} \sqrt{t}_* \check{H}_{n-1}\big) \right) dt,$$
at which point an application of \Cref{prop:ffdb int} yields the advertised inequality. 

All that remains to show is that $C_n >0$. We defer this routine calculation to \Cref{sec:cd neg}. 
\end{proof}

We remark that this inequality is saturated by Hermite polynomials, as the only inequality is in the application of \Cref{thm:fisher_monotonicity}, which itself is saturated by Hermite polynomials. 

\subsection{Lieb's inequality}

In this section, given any polynomial $r(x)$ of degree $n$ and $t\geq 0$,  we will use the notation 
$$r_t(x) := r(x) \boxplus_n \sqrt{t}_*\check{H}_n(x).$$ 

We will   establish the finite free version of Lieb's inequality following the arguments  in \cite{rioul2010information} without any change. The first step is to observe that, by the de Bruijn identity, the finite free Stam's inequality is equivalent to the monotonicity of the following function.

\begin{observation}
\label{obs:mono_from_stam}
      Let $p(x)$ and $q(x)$ be any monic real-rooted polynomials of degree $n$ and  $\gl\in[0,1]$. Then, the function 
      $$F(t):= \chi_n(\slambda p_t \boxplus_n \sone q_t ) - (\lambda \chi_n(p_t)+(1-\lambda)\chi_n(q_t))$$
      is non-increasing for all $t\geq 0$. 
\end{observation}

\begin{proof}
    Lemma \ref{thm:ffdb} and Theorem \ref{thm:stams_inequality} together imply that
    $$F'(t) = \frac{1}{2} \left( \Phi_n(\slambda p_t \boxplus_n \sone q_t)  - (\lambda\Phi_n(p_t)+(1-\lambda)\Phi_n(q_t)) \right)\leq 0,$$
    so the result follows. 
\end{proof}

Lieb's inequality follows readily. 

\begin{theorem}[Finite free Lieb inequality]\label{thm:ff7}
    Let $p(x)$ and $q(x)$ be any monic real-rooted polynomials of degree $n$ and  $\gl\in[0,1]$. Then
    \[\rchi_n\lpr{\sqrt{\gl}\,_*p\boxplus_n\sqrt{1-\gl}\,_*q}\ge\gl\rchi_n(p)+(1-\gl)\rchi_n(q).\]
\end{theorem}

\begin{proof}
Let $F(t)$ be as in Observation \ref{obs:mono_from_stam}.  Now note that, for any $t>0$ one can rewrite $F(t)$ as  
\begin{align*}
  &  \chi_n\big(\slambda p \boxplus_n \sone q  \boxplus_n \sqrt{t}_*\check{H}_n \big) - \big(\lambda \chi_n(p \boxplus_n \sqrt{t}_* \check{H}_n)+(1-\lambda)\chi_n(q \boxplus_n \sqrt{t}_*\check{H}_n)\big) 
\\ = &  \chi_n\left(\sqrt{t}_* \big( \sqrt{\lambda/t}_* p \boxplus_n (\sqrt{(1-\lambda)/t}_* q \boxplus_n \check{H}_n\big) \right) 
\\ & - \left( \lambda \chi_n\big(\sqrt{t}_* (\sqrt{1/t}_* p \boxplus_n \check{H}_n)\big) +(1-\lambda) \chi_n\big(\sqrt{t}_* (\sqrt{1/t}_* q \boxplus_n \check{H}_n) \big) \right) 
\\  = & \chi_n\left( \sqrt{\lambda/t}_* p \boxplus_n  (\sqrt{(1-\lambda)/t}_* q \boxplus_n \check{H}_n \right)  -  \left( \lambda \chi_n\big( \sqrt{1/t}_* p \boxplus_n \check{H}_n\big) +(1-\lambda) \chi_n\big( \sqrt{1/t}_* q\boxplus_n \hat{H}_n \big) \right) 
\end{align*}
where in the last equation we used the rescaling identity (\ref{eq:rescaling}). Now, from the last expression it is evident that $\lim_{t \to \infty}F(t)= 0$. Since $F(t)$ is non-increasing it follows that $F(0)\geq 0$, as we wanted to show.  
\end{proof}

\subsection{The finite free entropy power inequality}\label{sec:entropy_power}

We can now use the finite free Lieb inequality to derive the finite free entropy power inequality. This is standard and we follow the arguments in \cite{rioul2010information}.

\begin{proof}[Proof of Theorem \ref{thm:ffepi}]
  Take $\lambda\in [0, 1]$ such that 
  $$\lambda = \frac{\sN_n(p)}{\sN_n(p) + \sN_n(q)} \quad \text{and} \quad 1-\lambda =  \frac{\sN_n(q)}{\sN_n(p) + \sN_n(q)},$$
  and define $\tilde{p}:= \lambda^{-1/2}_* p$ and $\tilde{q} := (1-\lambda)^{-1/2}_* q$. Now, by the rescaling rules in (\ref{eq:rescaling}) we have that 
  $$\chi_n(\tilde{p}) = \chi_n(p)-\frac{1}{2}\log\left(  \frac{\sN_n(p)}{\sN_n(p) + \sN_n(q)} \right) = \frac{1}{2} \log(\sN_n(p)+\sN_n(q)),$$ 
  and similarly for $\chi_n(\tilde{q})$. On the other hand, Theorem \ref{thm:ff7} yields
  \[\rchi_n\lpr{\sqrt{\gl}\,_*\tilde{p}\boxplus_n\sqrt{1-\gl}\,_*\tilde{q}}\ge\gl\rchi_n(\tilde{p})+(1-\gl)\rchi_n(\tilde{q}).\]
  So, combining the above two expressions and using that $p = \slambda \tilde{p}$ and $q = \sone q$, we get 
  $$\chi_n(p\boxplus_n q) \geq \frac{1}{2} \log(\sN_n(p)+\sN_n(q)),$$
  and the proof is concluded by multiplying both sides by 2 and exponentiating. 
\end{proof}

\subsection{Equality cases}
\label{sec:equality_cases_epi}

We now determine the equality cases in Theorems \ref{thm:monotonicity},  \ref{thm:ffepi}, and \ref{thm:ff7}. 

\begin{theorem}
\label{thm:equality_cases_Lieb_EPI}
    Let $p(x)$ and $q(x)$ be monic degree $n$ real-rooted polynomials with simple roots. 
    \begin{enumerate}[i)]
     \item \label{item:equality_monotonicity} Let $C_n$ and $\widetilde{p'}$ be as in Theorem \ref{thm:monotonicity}. If $\rchi_n\lpr{p}+C_n=\rchi_{n-1}\lpr{\widetilde{p'}}$ then $p(x)$ is an affine transformation of $H_n(x)$. 
        \item \label{item:equality_Lieb} Let $\lambda\in (0, 1)$. If $\rchi_n\lpr{\sqrt{\gl}\,_*p\boxplus_n\sqrt{1-\gl}\,_*q}=\gl\rchi_n(p)+(1-\gl)\rchi_n(q)$ then $p(x)$ and $q(x)$ are affine transformations of $H_n(x)$. 
        \item \label{item:equality_EPI} If $\sN_n\lpr{p\boxplus_nq}=\sN_n(p)+\sN_n(q)$ then $p(x)$ and $q(x)$ are affine transformations of $H_n(x)$. 
    \end{enumerate}
\end{theorem}

\begin{proof}
   To get \ref{item:equality_monotonicity}) inspect the proof of Theorem \ref{thm:monotonicity} and note that the only inequality used was derived from the inequality in Theorem \ref{thm:fisher_monotonicity}. Moreover, if $\rchi_n\lpr{p}+C_n=\rchi_{n-1}\lpr{\widetilde{p'}}$ then equality most hold in Theorem \ref{thm:fisher_monotonicity}, so by Theorem \ref{thm:equality_cases_Fisher} we get that $p(x)$ must be the affine transformation of a Hermite polynomial. 
   
  Now we prove \ref{item:equality_Lieb}). Fix $\lambda\in (0, 1)$ and let $F(t)$ be as in Observation \ref{obs:mono_from_stam}. Then, the assumed equality translates into $F(0)=0$. On the other hand, we know that $F(t)$ is non-increasing and $\lim_{t\to \infty} F(t)=0$. So,  $F(0)=0$ implies $F'(0)=0$, but this  precisely translates into equality for Stam's inequality. So, by Theorem \ref{thm:equality_cases_Fisher} we can conclude that $p(x)$ and $q(x)$ are affine transformations of $H_n(x)$. 

   To prove \ref{item:equality_EPI}) simply note that, from the proof of Theorem \ref{thm:ffepi},  equality in the entropy power inequality implies equality for Lieb's inequality for a specific choice of $\lambda$ and some rescaling of the original polynomials, so the result follows from what we have proven in \ref{item:equality_Lieb}). 
\end{proof}

\section{Inequalities in the large-$n$ limit}
\label{sec:large_n_limit}

Here we re-derive some of the infinite-dimensional entropy inequalities already present in the literature \cite{szarek1996volumes, voiculescu1998analogues, shlyakhtenko2022fractional} by taking $n\to\infty$ in our inequalities. This provides fundamentally new proofs of these results. 

First, let us introduce some notation.  Given $\alpha\in \bR^n$, it will prove useful to define the quantity
$$\mesh(\alpha) := \min_{i\neq j} |\alpha_i - \alpha_j|, $$
and the probability measure 
$$\mu_\alpha := \frac{1}{n} \sum_{i=1}^n \delta_{\alpha_i}.$$
Conversely, given a probability measure $\mu$, its vector of $n$-quantiles $\alpha^n$ will defined as $\alpha^n =(\alpha^n_{1}, \dots, \alpha^n_{n})$ where 
$$\alpha^n_{i} := \inf \left\{x\in \bR : \mu(-\infty, x] \ge \frac{i}{n}\right\}. $$ 
\subsection{Free probability inputs}
\label{sec:free_proba_inputs}

When dealing with free entropy, we start by first deriving the  infinite-dimensional versions of our inequalities  for measures with bounded density,  and then extending these results to arbitrary measures. In order to do this we will need some knowledge of the regularizing effect of free convolutions. Specifically, we will use the following two results of Voiculescu. 

\begin{proposition}[Proposition 4.7 in \cite{voiculescu1993analogues}]
\label{prop:free_young_ineq}
    Let $\mu $ and $\nu$ be compactly supported probability measures on $\bR$. If $\mu$ is Lebesgue absolutely continuous with density $\rho_\mu$, then $\mu \boxplus \nu$ is also Lebesgue absolutely continuous and, its density $\rho_{\mu\boxplus \nu}$, satisfies 
    $$\|\rho_{\mu \boxplus \nu}\|_\infty \leq \|\rho_\mu\|_\infty. $$
\end{proposition}

\begin{proposition}[Proposition 6.4 a) in \cite{voiculescu1993analogues}]
\label{prop:continuity}
    Let $\mu$ be a compactly supported measure on $\bR$ and for $\eps>0$ let $\sigma_\eps$ denote the semicircle distribution of variance $\eps$. Then, for $\eps \in [0, \infty)$, $\eps \mapsto \chi(\mu \boxplus \sigma_\eps)$ is a continuous function of $\eps$. 
\end{proposition}

The proof of Proposition \ref{prop:continuity} is short and elementary (see the original source \cite{voiculescu1993analogues}). To prove Proposition \ref{prop:free_young_ineq} Voiculescu used subordination techniques. Interestingly, one can also obtain this result purely from finite free probability techniques. First, we establish the following. 

\begin{observation}
\label{obs:leb_abs_cont}
    Let $\mu$ be a compactly supported probability measure on $\bR$. Then, the following are equivalent: 
    \begin{enumerate}[(i)]
        \item \label{item:leb_abs_cont} $\mu$ is Lebesgue absolutely continuous and its density $\rho_\mu$ satisfies $\|\rho_\mu\|_\infty \leq B$. 
        \item \label{item:discrete_conv} There exist a sequence of vectors $\alpha^n \in \bR^n$ such that $\mu_{\alpha^n}$ converges weakly to $\mu$ and $\mesh(\alpha^n)\geq \frac{1}{nB}$ for every $n$. 
    \end{enumerate}
\end{observation}

\begin{proof}
    First assume (\ref{item:leb_abs_cont}). For every $n$, let $\alpha^n$ be the vector of $n$-quantiles of $\mu$. By definition, for every $i\in [n]$, we have that $\mu([\alpha_i^n, \alpha_{i+1}^n))=\frac{1}{n}$, which combined with the assumption in (\ref{item:leb_abs_cont}) yields $\alpha_i^{n}-\alpha_{i+1}^n\geq \frac{1}{n B}$, proving the claim in (\ref{item:discrete_conv}). 

    Now assume (\ref{item:discrete_conv}) and note that the hypothesis implies that $\mu_{\alpha^n}(a, b)\leq (b-a)B$ for all $a<b$.  This, in combination of the weak convergence of $\mu_{\alpha^n}$ y to $\mu$, yields $\mu (a, b )\leq B(b-a)$ for all $a< b$, which in turn implies (\ref{item:leb_abs_cont}). 
\end{proof}

We now prove Voiculescu's result. 

\begin{proof}[Finite free proof of Proposition \ref{prop:free_young_ineq}]
Let $(\alpha^n)_{n=1}^\infty$ and $(\beta^n)_{n=1}^\infty$ be sequences of vectors in $\bR^n$ such that $\mu_{\alpha^n}$ converges weakly to $\mu$ and $\mu_{\beta^n}$ converges weakly to $\nu$. Moreover, by Observation \ref{obs:leb_abs_cont} we may choose the $\alpha^n$ such that $\mesh(\alpha)\geq \frac{1}{n\|\rho_\mu\|}$. 

 Let $p_n(x):= \pmap{\alpha^n}(x)$ and $q_n(x) := \pmap{\beta_n}(x)$, and let $\gamma^n$ be the ordered vector of roots of $p_n \boxplus_n q_n$. Since $\mesh(\alpha)\geq \frac{1}{n\|\rho_\mu\|}$, from Theorem \ref{thm:preservation_of_mesh}  we get that $\mesh(\gamma)\geq \frac{1}{n\|\rho_\mu\|_\infty}$. On the other hand, since $\mu_{\alpha^n}$ and $\mu_{\beta^n}$ converge  weakly to   $\mu$ and $\nu$ respectively, by \cite[Corollary 5.5]{arizmendi2018cumulants} we have that $\mu_{\gamma_n}$ converges weakly to $\mu\boxplus \nu$. This, combined with $\mesh(\gamma)\geq \frac{1}{n\|\rho_\mu\|_\infty}$, by Observation \ref{obs:leb_abs_cont} yields the advertised result. 
\end{proof}

\begin{remark}[Future directions]
\label{rem:finite_free_young}
In view of the above proof one can argue that Theorem \ref{thm:preservation_of_mesh} is the finite version of Proposition \ref{prop:free_young_ineq}. On the other hand, in \cite[Proposition 4.7]{voiculescu1993analogues} Voiculescu proved that $\|\rho_{\mu \boxplus \nu}\|_p \leq \|\rho_\mu\|_p$ for all $1\leq p\leq \infty$, which is the free version of Young's inequality. So, although only $p=\infty$ is needed for our purposes, it is natural to ask if there are finite versions of Voiculescu's   free Young's inequality for all $1\leq p< \infty$. 
\end{remark}

\subsection{Free entropy inequalities}
\label{sec:limiting_free_entropy}

We now explain how to obtain the free entropy monotonicity of Shlyakhtenko and Tao \cite[Theorem 1.6]{shlyakhtenko2022fractional} and the free entropy power inequality of Szarek and Voiculescu \cite{szarek1996volumes}, by taking $n\to \infty$ in Theorems \ref{thm:monotonicity} and \ref{thm:ffepi}. We will need the following technical lemma. 

\begin{lemma}
\label{lem:convergence_lemma}
    Let $\mu$ be a compactly supported probability measure on $\bR$. Assume there exists a $B>0$ and a sequence of vectors  $(\alpha^n)_{n=1}^\infty$  with $\alpha^n \in \bR^n$ and such that $\mu_{\alpha^n}$ converges weakly to $\mu$ and $\mesh(\alpha^n)\geq \frac{1}{n B}$. If $p_n(x):= \pmap{\alpha^n}(x)$, then the following hold:
\begin{enumerate}[(i)]
    \item \label{item:entropy_convergence}  $\lim_{n\to \infty}\rchi_n\big(p_n\big) = \rchi(\mu)$. 
    \item \label{item:derivative_convergence} Let $t\in [0, 1)$ and let  $1\le k_n\le n$ is a sequence of positive integers with $\lim_{n\to \infty}\frac{k_n}{n} = t$, then $$\rchi_n\lpr{\left(\tfrac{1}{1-t}\right)_* \partial_x^{k_n} p_n(x) } = \chi(\mu^{\boxplus \frac{1}{1-t}}).$$
    \item \label{item:convolution_convergence} Let $\nu$ is another compactly probability measure on $\bR$ and  $(\beta_n)_{n=1}^\infty$ be a sequence of vectors such that $\beta^n \in \bR^n$ and $\mu_{\beta^n}$ converges weakly to $\nu$. Then, if $q_n(x) := \pmap{\beta_n}(x)$, we will have $$\lim_{n\to \infty}\rchi_n(p_n \boxplus_n q_n) = \rchi(\mu \boxplus \nu).$$ 
\end{enumerate}
\end{lemma}

\begin{proof} 
    We begin by proving \eqref{item:entropy_convergence}. For every $\varepsilon>0$, let 
     $$S_-(n, \varepsilon) := \sum_{i \neq j} \log\lvert\alpha_i^n-\alpha^n_j\rvert \cdot 1_{\{\lvert\alpha_i - \alpha_j\rvert\le \varepsilon\}}\quad \text{and} \quad S_+(n, \varepsilon) := \sum_{i\neq j} \log \lvert\alpha^n_i -\alpha^n_j\rvert \cdot 1_{\{\lvert\alpha^n_i-\alpha^n_j\rvert\ge \varepsilon\}},$$
     so that $\rchi_n(p_n) = \frac{2}{n(n-1)} (S_{-}(n, \varepsilon) + S_+(n, \varepsilon))$ for any $\varepsilon>0$. 

     Now, since by assumption $\mu_{\alpha^n}$ converges weakly to $\mu$,  for any fixed $\varepsilon>0$ we have
     $$\lim_{n\to \infty} \frac{2}{n(n-1)}S_+ (n, \varepsilon) = \iint \log\lvert s-t\rvert \cdot 1_{\{\lvert s-t\rvert\ge \varepsilon\}} \d\mu(s) \d\mu(t). $$
     Now, since the right-hand side of the above equation converges to $\rchi(\mu)$ as $\varepsilon\to 0$, and for every $\varepsilon$ we have 
     $\lim_{n\to \infty} \chi_n(p_n) = \lim_{n\to \infty} \frac{2}{n(n-1)} (S_{-}(n, \varepsilon) + S_+(n, \varepsilon)) $, to prove the claim it suffices to show that $\lim_{\varepsilon \to 0} \lim_{n\to \infty} \frac{2}{n(n-1)}S_-(n, \varepsilon) =0$. 

     Fix $\varepsilon>0$ and, to avoid trivial complications when dealing with logarithms, we will assume without loss of generality that $B\ge 1$.  Now, the assumption $\mesh(\alpha^n)\geq \frac{1}{n B}$ trivially implies
   \begin{equation}
   \label{eq:inverse_gap_estimate}
   \frac{1}{|\alpha^n_i - \alpha^n_j|} \le \frac{B n}{|i-j|}, \quad \forall i\neq j . 
      \end{equation}
      Now, for each $i\in [n]$ let $m_i(\varepsilon) := \max \{\lvert i-j\rvert : \lvert\alpha^n_i - \alpha^n_j\rvert \le \varepsilon  \} $ and note that 
   \begin{align}
 \nonumber   -\sum_{j : j \neq i} \log \lvert\alpha^n_i -\alpha^n_j\rvert \cdot 1_{\{\lvert\alpha^n_i - \alpha^n_j\rvert\le \varepsilon\}}   & \le \sum_{\lvert i-j\rvert \le m_i(\varepsilon)} \log \left( \frac{B n}{|i-j|}\right)  && \text{by (\ref{eq:inverse_gap_estimate})}
 \\ \nonumber & \leq 2 \sum_{\ell=1}^{m_i(\varepsilon)} \log \left( \frac{B n}{\ell} \right)
    \\\nonumber & = 2\log \left( \frac{(B n)^{m_i(\varepsilon)}}{m_i(\varepsilon)!} \right)  
    \\  \label{eq:log_bound} & \le 2m_i(\varepsilon)\log\left( \frac{B e n}{m_i(\varepsilon)} \right),
      \end{align}
   where in the last inequality we used the bound $k! \ge \left(\frac{k}{e}\right)^k$. On the other hand, the definition of $m_i(\varepsilon)$ together with \eqref{eq:inverse_gap_estimate} gives us that $m_i(\varepsilon) \le B \varepsilon n$ for all $n$. Now,  we use that for every fixed $r$ the sequence $k \log\frac{r}{k}$ is monotone increasing in $k$ for $1\le k\le\frac{r}{e}$ in conjunction with the fact that for $\varepsilon$ small enough one has  $B \varepsilon n\leq en $, to upper bound \eqref{eq:log_bound} by $2B \varepsilon n \log(e/\varepsilon)$. All put together, for small enough $\varepsilon>0$, yields
   $$-\frac{2}{n(n-1)}S_-(n,\eps)\leq \frac{2n}{n(n-1)} (2B \varepsilon n \log(e/\varepsilon)) = \frac{4B n}{n-1} \varepsilon  \log(e/\varepsilon) .$$
   It follows that $\lim_{\varepsilon \to 0} \lim_{n\to \infty} \frac{2}{n(n-1)} S_-(n, \varepsilon) =0$. 

  The claim in (\ref{item:derivative_convergence}) is proven in a similar way. Namely, let $\delta_1^n\geq \cdots \geq \delta_{n-k_n}^n$ be the roots of $\left( \frac{1}{1-t}\right)_* \partial_x^{k_n}p_n(x)$. Now,  Riesz's theorem \cite{stoyanoff1925theoreme} yields, $\mesh(\delta^n)\geq \mesh(\alpha^n)$, and since $\mesh(\alpha^n)\geq \frac{1}{nB}$, we get
  $$\frac{1}{|\delta_i^n - \delta_j^n|} \leq \frac{B n}{|i-j|}, \quad \forall i\neq j.$$
   On the other hand, since $\mu_{\alpha^n}$ converges weakly to $\mu$, by \cite[Theorem 4.4]{hoskins2023dynamics} we know that $\mu_{\delta^n}$ converges weakly to $\mu^{\boxplus \frac{1}{1-t}}$. At this point, the proof of (\ref{item:derivative_convergence}) follows exactly as the proof of (\ref{item:entropy_convergence}). 

   The proof of \eqref{item:convolution_convergence} is analogous.  Let $\gamma_{1}^n\le  \dots\le  \gamma_{n}^n$ be the roots of $p_n\boxplus_n q_n$. First, since $\mu_{\alpha^n}$ and $\mu_{\beta^n}$ converge weakly to $\mu$ and $\nu$, respectively, by \cite[Corollary 5.5]{arizmendi2018cumulants} we have that $\mu_{\gamma^n}$ converges weakly to $\mu\boxplus \nu$. Second, by Theorem \ref{thm:preservation_of_mesh} we have that $\mesh(\gamma^n)\geq \mesh(\alpha^n)$, and therefore
   \begin{equation}
   \label{eq:bound_on_gamma_difference}
   \frac{1}{|\gamma_{i}^n-\gamma_{j}^n|} \le \frac{B n}{|i-j|}, \quad \forall i \neq j.  
      \end{equation}
We can then replicate the proof of (\ref{item:entropy_convergence}) without any change to obtain the advertised result. 
\end{proof}

We can now recover the theorem of Szarek and Voiculescu \cite{szarek1996volumes}. 

\begin{theorem}[Szarek--Voiculescu]
\label{thm:szarek-voiculescu}
    Let $\mu$ and $\nu$ be any two compactly supported probability measures on $\bR$. Then, 
    $$\exp(2\rchi(\mu)) + \exp(2\rchi(\nu)) \le \exp(2\rchi(\mu\boxplus\nu)). $$
\end{theorem}

\begin{proof}
We will first prove the inequality in the case when $\mu$ and $\nu$ are compactly supported distributions with bounded density. In this case, let $(p_n)_{n=1}^\infty$ and $(q_n)_{n=1}^\infty$ be as in \Cref{lem:convergence_lemma}. By  \Cref{thm:ffepi} we have that $\exp(2\rchi_n(p_n))+\exp(2\rchi_n(q_n)) \le \exp(2\rchi_n(p_n\boxplus_n q_n))$. So, the claim is proven by taking $n\to \infty$ and invoking \Cref{lem:convergence_lemma}.  

To prove the claim for arbitrary compactly supported $\mu$ and $\nu$, use $\sigma_\varepsilon$ to denote the semicircular distribution of variance $\varepsilon$. By Proposition \ref{prop:free_young_ineq} we have that $\mu\boxplus \sigma_\varepsilon$ and $\nu\boxplus \sigma_\varepsilon$ are absolutely continuous with bounded density, so,  what we have already proven yields 
$$\exp(2\rchi(\mu \boxplus \sigma_\varepsilon))+ \exp(2\rchi(\nu \boxplus \sigma_\varepsilon)) \le \exp(2\rchi(\mu \boxplus \nu \boxplus \sigma_{2\varepsilon})).$$
On the other hand, by Proposition \ref{prop:continuity}, $\eps \mapsto \rchi(\mu \boxplus \sigma_\varepsilon)$ is continuous at $\varepsilon=0$. So, the proof is concluded by taking $\varepsilon\to 0$ in the above inequality.  
\end{proof}

The theorem of Shlyakhtenko and Tao \cite{shlyakhtenko2022fractional} can be obtained in the same way. Below, given a compactly supported probability measure $\mu$ on $\bR$ and $s\geq 1$ we will use $\mu^{\boxplus s}$ to denote the fractional free convolution of $\mu$. See \cite[\S 1.2]{shlyakhtenko2022fractional} for a discussion on this concept and \cite{nica1996multiplication} for the original source. As for polynomials, for $c>0$, we will use $c_* \mu$ to denote the probability measure obtained by rescaling $\mu$ by a factor of $c$.  

\begin{theorem}[Shlyakhtenko--Tao]
\label{thm:shlyakhtenko_tao_entropy}
    Let $\mu$ be any compactly supported probability measure on $\bR$. Then, for $s\geq 1$, $\chi\big(s^{-1/2}_* \mu^{\boxplus s}\big)$ is monotone non-decreasing in $s$. 
\end{theorem}

\begin{proof}
We want to show that if $s_1 >s_2 \geq 1$ then  $\chi\big((s_1)_*^{-1/2}\mu^{\boxplus s_1}\big) \geq \chi\big((s_2)_*^{-1/2}\mu^{\boxplus s_2}\big) $. But since, by definition of the fractional free convolution, $(\mu^{\boxplus a})^{\boxplus b}= \mu^{\boxplus ab} $ for all $a, b \geq 1$, note that it is enough to prove the statement in the case $s_1=1$. So, we fix $s\geq 1$ and aim to prove $\chi\big(s^{-1/2}_* \mu^{\boxplus s}\big)\geq \chi(\mu)$.

First assume that $\mu$ is absolutely continuous with bounded density.   Let $(p_n)_{n=1}^\infty$ be as in Lemma \ref{lem:convergence_lemma}, so that $\lim_{n\to \infty} \chi_n(p_n) = \chi(\mu)$. Now, by Theorem \ref{thm:monotonicity}, for every $n$ we have 
$\chi_n(p_n) < \chi_{n-1}\left( \sqrt{\tfrac{n-1}{n-2}}_* \partial_x p_n \right)$, and iterating this $k$ times gives
\begin{equation}
\label{eq:k_inequality}
\chi_n(p_n) < \chi_{n-k}\left( \sqrt{\tfrac{n-1}{n-k-1}}_* \partial_x^{k} p_n \right).
\end{equation}
So, if we take  $t$ such that  $\frac{1}{1-t}=s$ and $(k_n)_{n=1}^\infty$ so that $\lim_{n\to \infty }k_n/n = t$, when taking $n\to\infty$, Lemma \ref{lem:convergence_lemma} \ref{item:derivative_convergence}) in combination with (\ref{eq:k_inequality})  gives
$$\chi(\mu) \leq \chi(s^{-1/2}_* \mu^{\boxplus s}).$$
To obtain the claim for all compactly supported measures proceed as in the proof of Theorem \ref{thm:szarek-voiculescu} by applying Propositions \ref{prop:free_young_ineq} and Proposition \ref{prop:continuity}. 
\end{proof}

\subsection{Free Fisher information inequalities}

As for polynomials, given $c\geq 0$ and a compactly supported probability measure $\mu$, we will use $c_* \mu$ to denote the measure obtained by rescaling $\mu$ by a factor of $c$. And, as in  \cref{sec:limiting_free_entropy}, for $t \geq 0$ we will use $\sigma_t$ to denote the semicircular distribution of variance $t$. 

%\textcolor{red}{Define $c_*$ for measures, recall definition of $\sigma_t$.}

\begin{observation}
\label{obs:infinite_mono_stam}
Let $\mu$ and $\nu$ be two compactly supported absolutely continuous measures and fix $\lambda\in [0, 1]$. Then,   the function
$$F(t):= \chi\big(\slambda \mu \boxplus \sone \nu \boxplus \sigma_t \big) - \big(\lambda \chi\big(\mu\boxplus \sigma_t\big)+(1-\lambda)\chi\big(\nu\boxplus \sigma_t\big)\big)$$
is a non-increasing function for all $t\geq 0$. 
\end{observation}

\begin{proof}
By Observation \ref{obs:leb_abs_cont} we can take  sequences of vectors $\alpha^n, \beta^n \in \bR^n$ such that $\mu_{\alpha^n}$ and $\mu_{\beta^n}$ converge weakly, respectively, to $\mu$ and $\nu$, and that $\mesh(\alpha^n)\geq \frac{1}{nB}$ and $\mesh(\beta^n) \geq \frac{1}{nB}$ for some fixed constant $B>0$. Then set $p_n := \pmap{\alpha^n}$ and $q_n := \pmap{\beta^n}$ and for every  $t\geq 0$ let $p_{n, t}:= p_n \boxplus_n \sqrt{t}_* \check{H}_n$ and $q_{n, t} := q_n \boxplus_n \sqrt{t}_* \check{H}_n$. 

By Theorem \ref{thm:preservation_of_mesh} we have that $\mesh(p_{n, t})\geq \frac{1}{nB}$ and $\mesh(q_{n, t})\geq \frac{1}{nB}$. So, if we define $$F_n(t):= \chi_n(\slambda p_{n,t} \boxplus_n \sone q_{n, t} ) - (\lambda \chi_n(p_{n, t})+(1-\lambda)\chi_n(q_{n, t})),$$ 
by Observation \ref{obs:mono_from_stam}, for every $n$, $F_n(t)$ is non-increasing in $t$, while by Lemma \ref{lem:convergence_lemma} we have that
$$F(t) = \lim_{n\to \infty} F_n(t),$$
for any fixed $t\geq 0$. Now, since each $F_n(t)$ is non-increasing, the same holds true for $F(t)$. 
\end{proof}

As a corollary we can now obtain the free Stam inequality proven by Voiculescu in \cite{voiculescu1998analogues}. 

\begin{theorem}[Voiculescu]
\label{thm:infinite_free_stam}
    Let $\mu$ and $\nu$ be absolutely continuous compactly supported probability measures with bounded density. Then, for all $\lambda\in [0, 1]$, we have  
    $$\Phi(\slambda \mu \boxplus \sone \nu )\leq \lambda \Phi(\mu) + (1-\lambda)\Phi(\nu).$$
\end{theorem}

\begin{proof}
 Let $F(t)$ be as in Observation \ref{obs:mono_from_stam}. Then, by Voiculescu's free de Bruijn's identity \cite[Lemma 3.2]{voiculescu1993analogues} we know that $F(t)$ is differentiable and 
 $$F'(0) = \frac{1}{2}\big(\Phi(\slambda \mu \boxplus \sone \nu)-(\lambda\Phi(\mu)+(1-\lambda)\Phi(\nu))\big).$$
 On the other hand, by Observation \ref{obs:infinite_mono_stam}, we have that $F'(0)\leq 0$, so the claim follows. 
\end{proof}

With a very similar argument we can use the monotonicity for the finite free entropy and Fisher information to obtain the following result of Shlyakhtenko and Tao \cite{shlyakhtenko2022fractional}. 

\begin{theorem}[Shlyakhtenko--Tao]
\label{thm:infinite_free_mono}
    Let $\mu$ be an absolutely continuous compactly supported probability measure on $\bR$. Then, for $s\geq 1$, $\Phi(s_*^{-1/2} \mu^{\boxplus s})$ is monotone non-increasing in $s$. 
\end{theorem}

\begin{proof}
Again, as it was argued in the proof of Theorem \ref{thm:shlyakhtenko_tao_entropy}, it is enough to  show that for any fixed $s\geq 1$ one has that $\Phi(s_*^{-1/2} \mu^{\boxplus s})- \Phi(\mu)\leq 0$. And this can be established via  Voiculescu's free de Bruijn's identity \cite[Lemma 3.2]{voiculescu1993analogues} by showing that 
$$G(t):=\chi(s_*^{-1/2} \mu^{\boxplus s}\boxplus \sigma_t)- \chi(\mu \boxplus \sigma_t),$$ is a non-increasing function of $t$. 

Now take $\alpha^n\in \bR^n$ such that $\mu_{\alpha^n}$ converges weakly to $\mu$ and that $\mesh(\alpha^n)\geq \frac{1}{nB}$ for some fixed constant $B>0$. Then set $p_n:= \pmap{\alpha^n}$. Now, take a sequence of integers $k_n$ such that $\lim_{n\to \infty }k_n/n=1-\frac{1}{s}$ and put 
$$q_n:=\sqrt{\tfrac{n-1}{n-k_n-1}}_* \partial_x^{k_n} p_n, $$
so that the root distribution of $q_n$ converges weakly to $s_*^{-1/2} \mu^{\boxplus s}$. Then, again by Lemma \ref{lem:convergence_lemma} it will be enough to show that the functions
$$G_n(t):= \chi_{n-k_n}(q_n \boxplus_n \sqrt{t}_* \check{H}_{n-k_n}) - \chi_n(p_n\boxplus_n \sqrt{t}_* \check{H}_n),$$
are a non-increasing function of $t$. The proof then follows by differentiating $G_n$ with respect to $t$ and applying Lemma \ref{thm:ffdb} in conjunction with (\ref{eq:diff_p_t}) and Theorem \ref{thm:fisher_monotonicity} to conclude that $G_n(t)$ is non-increasing for every $n$.  
\end{proof}

\begin{remark}[Subtleties in the limiting behavior of $\Phi_n$]
    Unlike  the case of the free entropy (discussed in  \cref{sec:limiting_free_entropy}) where the convergence of $\chi_n$ to $\chi$ holds as soon as weak convergence of the root distributions and a mesh condition are satisfied, the convergence of $\Phi_n$ to $\Phi$ is much more delicate. 
    
    In fact, for any compact absolutely continuous measure $\mu$, it is not hard to find examples of sequences of vectors $\alpha^n \in \bR^n$ with $\mesh(\alpha^n)  = \Omega(\frac{1}{n})$ and $\mu_{\alpha^n}$ converging weakly to $\mu$,  such that $\lim_{n\to \infty}\Phi_n(\pmap{\alpha^n}) \neq \Phi(\mu)$. In a nutshell, the discrepancy between the limiting behavior of $\chi_n$ and $\Phi_n$, resides in the fact that for $\alpha\in \bR^n$ with $\mesh(\alpha)=\Theta(\frac{1}{n})$, perturbations of individual entries of $\alpha$ that keep the mesh at order $\Theta(\frac{1}{n})$   have a $\Theta(\frac{1}{n})$ effect on $\Phi_n$ while only having an order $\Theta(\frac{1}{n^2})$ on $\chi_n$. So, by perturbing $\Omega(n)$ entries of $\alpha$ a microscopic amount one can macroscopically change the value of $\Phi_n$. So, even if for any given  compactly supported measure $\mu$ with bounded density, it is possible to carefully construct sequences of $\alpha^n\in \bR^n$ with $\mu_{\alpha^n}$ converging weakly to $\mu$ and $\lim_{n\to \infty} \Phi_n(\pmap{{\alpha^n}}) = \Phi(\mu)$, once one takes finite free convolutions or applies some other operation to $\alpha^n$, it is hard to control the limiting behavior of the finite free Fisher information of the resulting vectors. 

    For the above reason, we did not immediately see a way to directly derive Theorems \ref{thm:infinite_free_stam} and \ref{thm:infinite_free_mono} from Theorems \ref{thm:stams_inequality} and \ref{thm:fisher_monotonicity}, and ultimately chose the  quicker route given above. 
\end{remark}

\printbibliography

\appendix

\section{Proof that $C_n>0$ in the entropy monotonicity}\label{sec:cd neg}

Recall that $C_n=\rchi_{n-1}(\check{H}_{n-1})-\rchi_n(\check{H}_n)$. We have the relationship that 
\begin{equation*}
    \rchi_n(p)=\frac{1}{n(n-1)}\log\abs{\Disc p}%\label{eq:chid as disc}
\end{equation*}
when $p$ is monic of degree $n$; applied to $\check{H}_n=\frac{1}{\sqrt{n-1}}_*H_n$, whose discriminant is given by the recursion $H_{n+1}(x)=xH_n(x)-nH_{n-1}(x)$ and the explicit form found in \cite[6.71.2]{szego1975}, we obtain: 
\iffalse
Let $\He_n$ be the $n$-th physicist's Hermite polynomial: $H_n=\frac{1}{2^n}\cdot\sqrt{2}_*\He_n$. Then, $\check{H}_n=\frac{1}{2^n}\cdot\sqrt{\frac{2}{n-1}}_*\He_n$. We have the relationship that 
\begin{equation}
    \rchi_n(p)=\frac{1}{n(n-1)}\log\abs{\frac{\Disc p}{a^{2(n-1)}}}\label{eq:chid as disc}
\end{equation}
where $p$ has degree $n$ and $a$ is $p$'s leading coefficient. So, 
\begin{equation}\label{eq:chid Htilded}
    \rchi_n(\check{H}_n)=\rchi_n\lpr{\frac{1}{2^n}\cdot\sqrt{\tfrac{2}{n-1}}_*\He_n}=\rchi_n\lpr{\sqrt{\tfrac{2}{n-1}}_*\He_n}=\tfrac{1}{2}\log\tfrac{2}{n-1}+\rchi_n(\He_n)
\end{equation}
by \eqref{eq:rescaling}. $\He_n$ has leading coefficient $2^n$, and per \cite{mathworld}, $\Disc\He_n=2^{\frac{3}{2}n(n-1)}\prod\limits_{k=1}^nk^k$. From \eqref{eq:chid as disc}, 
\begin{equation}\label{eq:chid Hed}
    \rchi_n(\He_n)=-\frac{1}{2}+\frac{1}{n(n-1)}\sum\limits_{k=1}^nk\log k.
\end{equation}
Combining \eqref{eq:chid Htilded} and \eqref{eq:chid Hed} gives: 
\fi
\begin{align}
    \rchi_n(\check{H}_n)&=\frac{1}{2}\log\frac{2}{n-1}-\frac{1}{2}+\frac{1}{n(n-1)}\sum\limits_{k=1}^nk\log k\nonumber\\
    \rchi_{n-1}(\check{H}_{n-1})&=\frac{1}{2}\log\frac{2}{n-2}-\frac{1}{2}+\frac{1}{(n-1)(n-2)}\sum\limits_{k=1}^{n-1}k\log k\nonumber\\
    \rchi_n(\check{H}_n)-\rchi_{n-1}(\check{H}_{n-1})&=\frac{1}{2}\log\frac{n-2}{n-1}+\frac{1}{n-1}\log n-\frac{2}{n(n-1)(n-2)}\sum\limits_{k=1}^{n-1}k\log k\\
    &<\frac{1}{2}\log\frac{n-2}{n-1}+\frac{1}{n-1}\log n-\frac{2}{n(n-1)(n-2)}\int_1^{n-1}x\log x\d x\nonumber\\
    &=\frac{1}{2}\log\frac{n-2}{n-1}+\frac{1}{n-1}\log n -\frac{1}{2\ln 2\cdot n(n-1)(n-2)}\nonumber\\
    & \quad -\frac{n-1}{n(n-2)}\log(n-1).\label{eq:eval integral}
\end{align}
We use the Taylor series 
\begin{align}
    -\log\left(1-\frac{1}{n}\right)&=\sum\limits_{k=1}^\infty\frac{1}{kn^k}<\sum\limits_{k=1}^\infty\frac{1}{n^k}=\frac{1}{n-1}\nonumber\\
    \log\frac{n-2}{n-1}=\log\left(1-\frac{1}{n-1}\right)&=-\sum\limits_{k=1}^\infty\frac{1}{k(n-1)^k}<-\sum\limits_{k=1}^\infty\frac{1}{2^{k-1}(n-1)^k}=-\frac{1}{n-\frac{3}{2}}\nonumber
\end{align}
so that we may continue to find that
\begin{equation}
    \eqref{eq:eval integral}<\frac{p(n)}{2\ln2\cdot n(n-1)\left(n-\frac{3}{2}\right)(n-2)}-\frac{1}{n(n^2-3n+2)}\log n\label{eq:final expr}
\end{equation}
where $p(x)=(1-2\ln2)x^3+\left(4\ln2-\frac{7}{2}\right)x^2+(3+\ln2)x+3\ln2$. We numerically evaluate that $p$ has largest real root at $x\approx1.78$. Therefore, as $p$'s leading coefficient is negative, it is evident that for $n\ge3$, the right-hand side of \eqref{eq:final expr} is negative. 

\end{document}